\documentclass[a4paper,12pt]{article}
\usepackage{amsthm}
\usepackage{amsmath,amsfonts,amssymb}
\usepackage{a4wide}
\usepackage[usenames,dvipsnames]{color}
\usepackage[all,cmtip]{xy}
\usepackage[verbose,colorlinks=true,linktocpage=true,linkcolor=blue,citecolor=blue]{hyperref}
\usepackage{footnote}
\usepackage{extarrows}%2023.01.25 arrows

\usepackage{tikz}
\usetikzlibrary{arrows.meta}

\theoremstyle{definition}
\newtheorem{definition}{Definition}[section]

\theoremstyle{plain}
\newtheorem{theorem}[definition]{Theorem}
\newtheorem{proposition}[definition]{Proposition}
\newtheorem{lemma}[definition]{Lemma}

\newtheorem{example}[definition]{Example}
\theoremstyle{remark}
\newtheorem{remark}[definition]{Remark}

\numberwithin{equation}{section}

\begin{document}
\title{Braid group action and quantum affine superalgebra for type $\mathfrak{osp}(2m+1|2n)$}

\author{Xianghua Wu${}^{1,2}$, Hongda Lin${}^{1,3}$  and Honglian Zhang${}^{1,}$\thanks{ Corresponding author \quad Email: hlzhangmath@shu.edu.cn}}
\maketitle

\begin{center}
\footnotesize
\begin{itemize}
\item[1] Department of Mathematics, Shanghai University, Shanghai 200444, China.
\item[2] School of Mathematics and Computing Science, Guilin University of Electronic Technology, Guilin 541000,  P.R. China.
\item[3] Shenzhen International Center for Mathematics, Southern University of Science and Technology, Guangdong 518055, P.R. China.
\end{itemize}
\end{center}
\begin{abstract}

In this paper, we investigate the structure of the quantum affine superalgebra associated with the orthosymplectic Lie superalgebra $\mathfrak{osp}(2m+1|2n)$ for $m\geqslant 1$. The Drinfeld-Jimbo presentation for this algebra, denoted as $U_q[\mathfrak{osp}(2m+1|2n)^{(1)}]$, was originally introduced by H. Yamane. We provide the definition of the Drinfeld presentation $\mathcal{U}_q[\mathfrak{osp}(2m+1|2n)^{(1)}]$.
To establish the isomorphism between the Drinfeld-Jimbo presentation and the Drinfeld presentation of the quantum affine superalgebra for type $\mathfrak{osp}(2m+1|2n)$, we introduce a braid group action to define quantum root vectors of the quantum superalgebra.
Specifically, we present an efficient method for verifying the isomorphism between two presentations of the quantum affine superalgebra associated with the type $\mathfrak{osp}(2m+1|2n)$.\\


\end{abstract}
%
%\MRSubClass{05B05}      % MR(2000) Subject Classification

%\baselineskip 15pt

\section{Introduction}
Quantum affine algebras occupy a significant position in the realm of quantum groups and play a pivotal role in various aspects of mathematical physics, particularly in the study of solvable systems and their associated symmetries (\cite{BOMTA}). Independently introduced by Drinfeld (\cite{VGD}) and Jimbo (\cite{MJB}), they represent $q$-deformations of the universal enveloping algebras of affine Kac-Moody algebras. Subsequently, Drinfeld (\cite{VGD2}) presented a new presentation of quantum affine algebras, commonly known as the Drinfeld presentation. This presentation has been demonstrated to offer a clearer framework for studying the representation theory of quantum affine algebras (see \cite{VCA}).

{
In the untwisted case, Beck (\cite{BECK2}) pioneered the construction of a surjective homomorphism from the Drinfeld-Jimbo presentation to the Drinfeld presentation, crucially utilizing Lusztig's braid group action (\cite{Lusz1, Lusz2}). Subsequently, Beck (\cite{BECK1}) extended this methodology to establish a Poincar  -Birkhoff-Witt (PBW)-type convex basis for quantum affine algebras via braid group operators, and a result specialized to the $\widehat{\mathfrak{sl}}_2$ case by Damiani (\cite{Dami1}) and later generalized to twisted types in \cite{Dami2}. Notably, Damiani \cite{Damiani1,Damiani2} rigorously demonstrated that Beck's surjective homomorphism is indeed an isomorphism for both untwisted and twisted quantum affine algebras.
In parallel developments for twisted algebras, Jing-Zhang (\cite{NJHZ}) formulated the Drinfeld realization and proved its equivalence to the Drinfeld-Jimbo presentation.
}

As a supersymmetric extension of quantum affine algebras, quantum affine superalgebras were introduced by incorporating extra generators and relations to account for the $\mathbb{Z}_2$-grading. Over the past two decades, these algebraic structures have garnered significant interest from mathematicians and physicists, driven largely by their applications in exactly solvable lattice models in statistical physics and knot theory. In \cite{HYAM}, Yamane provided the Drinfeld-Jimbo presentations for the quantum affine superalgebras of type $A$-$G$  and extended Beck's method (\cite{BECK2}) to construct the Drinfeld presentation specifically {for $U_q(\widehat{\mathfrak{sl}}_{m|n})$. On the other hand, Cai-Wang-Wu-Zhao (\cite{JSKW}), Fan-Hou-Shi (\cite{HBKD}), and Zhang (\cite{YZC}) have constructed the Drinfeld presentation of the quantum affine superalgebra for $U_q(\widehat{\mathfrak{gl}}_{m|n})$} by applying the Ding-Frenkel theorem (\cite{JDIB}).

Given the increased complexity of the structure of quantum affine superalgebras compared to the non-super case, many scholars tend to focus solely on the type A quantum affine superalgebra. Stukopin (\cite{VS}) investigated the connection between type $A$ super $\hbar$-Yangian and the quantum loop superalgebra, adopting Gautam-Toledano Laredo's approach (\cite{GT}). Tsymbaliuk (\cite{AT}) established a basis of Drinfeld generators of PBW type for the quantum loop superalgebra $U_v{L\mathfrak{sl}(m|n)}$, which can be extended in parallel to the quantum affine superalgebra $U_v{\widehat{\mathfrak{sl}}(m|n)}$. Established this PBW basis for the quantum affine superalgebra of type A, the authors in (\cite{BCFI}) introduced the quantum generalized imaginary Verma modules induced from an irreducible module of quantum Heisenberg algebras. Moreover, numerous works have focused on the representation theory, particularly in terms of the Drinfeld presentation, for type A, such as \cite{KKJSJ}, \cite{ZHFG}, \cite{ZR1}, \cite{YZC2}, and others.

Therefore, delving deeper into the structure and representation theory of quantum affine superalgebras for other cases remains an intriguing and pivotal area of study, notably in the pursuit of formulating the Drinfeld presentation for these superalgebras. Nevertheless, there exist only a few studies focusing on the Drinfeld presentation of quantum affine superalgebras beyond type A. Examples include $U_q[\mathfrak{osp}(1|2n)^{(1)}]$, $U_q[\mathfrak{osp}(2|2n)^{(2)}]$, $U_q[\mathfrak{sl}(1|2n)^{(2)}]$ ({refer to \cite{DGLZ, YXRZ2})} and $U_q[D(2,1;x)^{(1)}]$ (refer to \cite{HSTY}).

In this paper, our objective is to establish the Drinfeld presentation of the quantum affine superalgebra associated with the orthosymplectic Lie superalgebra $\mathfrak{osp}(2m+1|2n)$ while excluding the case $m=0$. We specifically focus on $\mathfrak{osp}(2m+1|2n)^{(1)}$ $(m\geqslant 1)$ with the following standard Dynkin diagram ($N=m+n$):
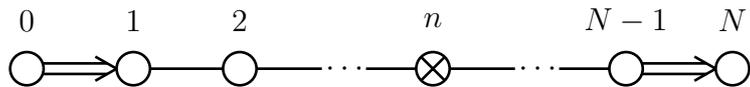
\begin{center}
    \begin{tikzpicture}
    \draw (0.22,0) circle (0.22);
    \node[above] at (0.22,0.4) {0};
    \draw[double, double distance=3pt] (0.44,0) -- (1.1,0);
    \draw[-{To[scale=2, line width=0.4pt]}, line width=0.4pt] (1.2,0) -- ++(0.01,0);
    \draw (1.46,0) circle (0.22);
    \node[above] at (1.46,0.4) {1};
    \draw (1.68,0) -- (2.48,0);
    \draw (2.70,0) circle (0.22);
    \node[above] at (2.70,0.4) {2};
    \draw (2.92,0) -- (3.62,0);
    \node at (4.00,0) {$\cdots$};
    \draw (4.30,0) -- (4.82,0);
    \draw (5.04,0) circle (0.22) (4.88,0.16) -- (5.2,-0.16) (4.88,-0.16) -- (5.2,0.16);
    \node[above] at (5.04,0.4) {$n$};
    \draw (5.26,0) -- (5.96,0);
    \node at (6.34,0) {$\cdots$};
    \draw (6.64,0) -- (7.16,0);
    \draw (7.38,0) circle (0.22);
    \node[above] at (7.38,0.4) {$N-1$};
    \draw[double, double distance=3pt] (7.60,0) -- (8.26,0);
    \draw[-{To[scale=2, line width=0.4pt]}, line width=0.4pt] (8.36,0) -- ++(0.01,0);
    \draw (8.62,0) circle (0.22);
    \node[above] at (8.62,0.4) {$N$};
\end{tikzpicture}
\end{center}

As noted, Levendorskii (\cite{SZLO}) provided a `minimal presentation' with finitely many Drinfeld generators and relations for Yangians. Jing and Zhang (\cite{NJHZ2}) and Gao, Jing, Xia, and Zhang (\cite{GJXZ}) extended this method to quantum toroidal algebras, and also presenting a `minimal presentation' from their Drinfeld presentations. Recently, Lin, Yamane, and Zhang (\cite{LYZ}) provided a detailed proof of the isomorphism between the Drinfeld-Jimbo presentation and the Drinfeld presentation of the quantum affine superalgebra $U_q(\widehat{\mathfrak{sl}}_{m|n})$ via the `minimal presentation'. {Here we apply this efficient method again,} combined with braid group actions, to assert that this newly formed `minimal presentation' is isomorphic to both the Drinfeld and Drinfeld-Jimbo presentations of the quantum affine superalgebra associated with $\mathfrak{osp}(2m+1|2n)$.

The paper is organized as follows. In Section 2, we introduce notations and the root system of the orthosymplectic affine Lie superalgebra $\mathfrak{osp}(2m+1|2n)^{(1)}$. In Section 3, we provide a review of the Drinfeld-Jimbo presentation of the quantum affine superalgebra $U_q[\mathfrak{osp}(2m+1|2n)^{(1)}]$, define quantum root vectors in the context of \textit{Chevalley generators}. Additionally, we establish the Drinfeld presentation of the quantum affine superalgebra $\mathcal{U}_q[\mathfrak{osp}(2m+1|2n)^{(1)}]$. The main result of this paper is the demonstration of an algebraic isomorphism between these two superalgebras. In Section 4, we delve into the braid group action on $U_q[\mathfrak{osp}(2m+1|2n)^{(1)}]$, presenting essential lemmas contributing to our main theorem. In Section 5, we furnish the proof of our primary theorem, employing a `minimal presentation' denoted as $\mathcal{U}'$.

\section{Preliminaries}
Let $\mathbb{C}$ denote the field of complex numbers, $\mathbb{Q}$ represent the field of rational numbers, $\mathbb{Z}$ stand for the commutative group of integers, and $\mathbb{Z}_+$ refer to the semigroup of non-negative integers. Additionally, consider $\mathbb{Z}_2=\{\bar{0},\bar{1}\}\simeq \mathbb{Z}/2\mathbb{Z}$ as the quotient group.

A superalgebra $X=X_{\bar{0}}\oplus X_{\bar{1}}$ is a $\mathbb{Z}_2$-graded algebra over $\mathbb{C}$, satisfying $X_{\bar{i}}X_{\bar{j}}\subseteq X_{\overline{i+j}}$ for $i,j\in \mathbb{Z}_2$. In this context, elements of $X_{\bar{0}}$ are referred to as even, while those of $X_{\bar{1}}$ are termed odd. The parity of a homogeneous element $x\in X_{\bar{i}}$ is denoted by $[x]=i$. Let $X,Y$ be two superalgebras over $\mathbb{C}$. Then, the \textit{tensor product} $X\otimes Y$ forms a superalgebra with the following multiplication rule for homogeneous elements $x_1,x_2\in X$, $y_1,y_2\in Y$:
\begin{gather*}
(x_1\otimes y_1)(x_2\otimes y_2)=(-1)^{[y_1][x_2]}(x_1x_2\otimes y_1y_2).
\end{gather*}

A Lie superalgebra $\mathfrak{g}=\mathfrak{g}_{\bar{0}}\oplus\mathfrak{g}_{\bar{1}}$ is a $\mathbb{C}$-superspace equipped with a bilinear operator $[\,\cdot\,,\,\cdot\,]:\mathfrak{g}\times \mathfrak{g}\rightarrow \mathfrak{g}$, subject to the following conditions for all homogeneous elements $x,y,z\in\mathfrak{g}$:
\begin{equation*}
\begin{array}{ll}
 &[x,\,y]=-(-1)^{[x][y]}[y,\,x], \\ &[x,\,[y,\,z]]=[[x,\,y],\,z]+(-1)^{[x][y]}[y,\,[x,\,z]].
\end{array}
\end{equation*}
For the orthosymplectic Lie superalgebra $\mathfrak{osp}(2m+1|2n)$ with $m\geqslant 1$, we introduce the $\mathbb{Z}_2$-value function $[~\cdot~]$ on the index set $\{1,\ldots,n,n+1,\ldots,m+n\}$. Particularly, we call $\mathfrak{osp}(2m+1|2n)$ standard if
\begin{equation*}
  [i]=\begin{cases}
  1,&1\leqslant i\leqslant n,\\
  0,&n+1\leqslant i\leqslant m+n.
  \end{cases}
\end{equation*}

Let $\mathfrak{h}$ denote the Cartan subalgebra of $\mathfrak{osp}(2m+1|2n)$. We choose a set of linearly independent functions $\{\varepsilon_i | 1 \leqslant i \leqslant m+n\}$ on $\mathfrak{h}$ such that
\begin{gather*}
\mathfrak{h}^{\ast} = \operatorname{span}_{\mathbb{C}}\{\varepsilon_i | 1 \leqslant i \leqslant n+m\},
\end{gather*}
where the symmetric bilinear form $(\,\cdot\,,\,\cdot\,)$ on $\mathfrak{h}^{\ast}$ is defined as $(\varepsilon_i, \varepsilon_j) = (-1)^{[i]}\delta_{ij}$.
For the simple roots of $\mathfrak{osp}(2m+1|2n)$, consider:
\begin{align*}
&\alpha_i = \varepsilon_i - \varepsilon_{i+1} \quad \text{for }1 \leqslant i < n+m, \\
&\alpha_{n+m} = \varepsilon_{n+m},
\end{align*}
where $[\alpha_i] = 0$ for $i \neq n$, and $[\alpha_n] = 1$. Let $\Pi = \{\alpha_1, \ldots, \alpha_{n+m}\}$ be the root base, and $\theta = 2\sum_{i=1}^{n+m}\alpha_i = 2\varepsilon_1$ be the highest root. The set $\Pi$ spans $\mathfrak{h}^{\ast}$.
Let $Q_+ = \mathbb{Z}_{+}\alpha_1 \oplus \cdots \oplus \mathbb{Z}_{+}\alpha_{m+n}$, and $Q = Q_+ \cup -Q_+$. The set $Q$ is referred to as the root lattice. Denote $\Delta$ (resp. $\Delta_+$) as the root system (resp. the positive root) of $\mathfrak{osp}(2m+1|2n)$.
It is commonly acknowledged that the sets of even positive roots and odd positive roots are given by:
\begin{align*}
&\Delta_+^{\bar{0}} = \{\varepsilon_i \pm \varepsilon_j, 2\varepsilon_k, \varepsilon_l \,|\, [i]=[j], 1 \leqslant i < j \leqslant n+m, [k]=1, [l]=0\}, \\
&\Delta_+^{\bar{1}} = \{\varepsilon_i \pm \varepsilon_j, \varepsilon_k \,|\, [i]\neq [j], 1 \leqslant i < j \leqslant n+m, [k]=1\},
\end{align*}
where $\Delta_+ = \Delta_+^{\bar{0}} \cup \Delta_+^{\bar{1}}$.

Furthermore, we have the following decomposition{
\[
\mathfrak{osp}(2m+1|2n) = \mathfrak{h} \oplus \left(\oplus_{\alpha \in \Delta} \mathfrak{osp}(2m+1|2n)_{\alpha}\right)
\]
where $\mathfrak{osp}(2m+1|2n)_{\alpha}$ is the root space such that $\dim \mathfrak{osp}(2m+1|2n)_{\alpha}=1$.} The Cartan matrix $A=(A_{ij})_{i,j=1}^{m+n}$ of $\mathfrak{osp}(2m+1|2n)$ is defined as:
\begin{equation*}
  A_{ij} = \begin{cases}
  (\alpha_i, \alpha_j), & i < m+n,\\
  2(\alpha_i, \alpha_j), & i = m+n.\\
  \end{cases}
\end{equation*}
Let $h_i \in \mathfrak{h}$ be such that $\langle h_i, \alpha_j \rangle = A_{ij}$ for all $j=1, \ldots, n+m$, {where $\langle, \rangle$ is the Killing form} of Lie superalgebra. Set $\varpi_i = |(\alpha_i, \alpha_i)|/2$. There exists a supersymmetric invariant bilinear form $(\,\cdot\,,\,\cdot\,)$ on $\mathfrak{osp}(2m+1|2n)$ satisfying the following conditions:

\vspace{6pt}
(1) The restriction of $(\,\cdot\,,\,\cdot\,)$ on $\mathfrak{h}$ is non-degenerate, and $(h_i, h) = \varpi_i \langle h, \alpha_i \rangle$ for all $h \in \mathfrak{h}$ and $i = 1, \ldots, m+n$.

(2) For $\alpha \in \Delta^+$, the restriction of $(\,\cdot\,,\,\cdot\,)$ on {$\mathfrak{osp}(2m+1|2n)_{\alpha} \oplus \mathfrak{osp}(2m+1|2n)_{-\alpha}$} is non-degenerate. More precisely, $[x_{\alpha}, y_{\alpha}] = (x_{\alpha}, y_{\alpha})h_{\alpha}$ for {$x_{\alpha} \in \mathfrak{osp}(2m+1|2n)_{\alpha}$, $y_{\alpha} \in \mathfrak{osp}(2m+1|2n)_{-\alpha}$,} where $h_{\alpha}$ satisfies $(h_{\alpha}, h) = \varpi_i \langle h, \alpha \rangle$ for all $h \in \mathfrak{h}$.

(3) {$(\mathfrak{osp}(2m+1|2n)_{\alpha},\mathfrak{osp}(2m+1|2n)_{\beta}) = 0$,} if $\alpha \neq -\beta$.

\vspace{6pt}
Now, let us examine the affine Lie superalgebra $\mathfrak{osp}(2m+1|2n)^{(1)}$. As the notation in \cite{Vdl}, we have
\begin{gather*}
\mathfrak{osp}(2m+1|2n)^{(1)} = \left(\mathfrak{osp}(2m+1|2n)\otimes_{\mathbb{C}}[t,t^{-1}]\right)\oplus \mathbb{C}c\oplus \mathbb{C}d,
\end{gather*}
where the Cartan subalgebra is denoted as $\hat{\mathfrak{h}} = \mathfrak{h}\oplus \mathbb{C}c\oplus \mathbb{C}d$. Here, the element $c$ is central, and the element $d$ serves as the even superderivation, acting on $\mathfrak{osp}(2m+1|2n)^{(1)}$ according to
\begin{gather*}
[d,\,x\otimes t^r] = rx\otimes t^r
\end{gather*}
for homogeneous elements $x\in \mathfrak{osp}(2m+1|2n)$ and $r\in\mathbb{Z}$.
Furthermore, for $x,y\in \mathfrak{osp}(2m+1|2n)$, the commutator on {$(\mathfrak{osp}(2m+1|2n)\otimes_{\mathbb{C}}[t,t^{-1}])\otimes\mathbb{C}c$  }is defined as
\begin{gather*}
[x\otimes t^r,\,y\otimes t^s] = [x,\,y]\otimes t^{r+s} + \delta_{r,-s}r(x,y)c.
\end{gather*}

\vspace{6pt}
Consider $\delta\in\hat{\mathfrak{h}}^{\ast}$ as the linear function satisfying $\langle d, \delta \rangle=1$ and $\langle c, \delta \rangle = \langle \mathfrak{h}, \delta \rangle=0$. The set
\begin{gather*}
\{\alpha_0:=\delta-\theta, \alpha_1, \ldots, \alpha_{n+m}\}
\end{gather*}
denotes the affine root base $\hat{\Pi}$ of $\mathfrak{osp}(2m+1|2n)^{(1)}$.
The affine Cartan matrix $\hat{A}$ of standard $\mathfrak{osp}(2m+1|2n)^{(1)}$ is obtained by adjoining the 0th row and column such that
\begin{gather*}
A_{00}=-2, \quad A_{10}=2A_{01}=-2, \quad A_{j0}=A_{0j}=0\quad \text{for } 1<j\leqslant n+m.
\end{gather*}
In the context of $\mathfrak{osp}(2m+1|2n)^{(1)}$, we use the symbols $\widehat{\Delta}$, $\widehat{\Delta}_+$, and $\widehat{Q}$ to represent the affine root system, affine positive root system, and root lattice, respectively. {
Let $\underline{\Delta}_+$ denote the reduced root system obtained from the positive root system $\Delta_+$ of $\mathfrak{osp}(2m+1|2n)$ by excluding roots $\alpha$ such that $\alpha/2$ are odd roots. Here we use the standard notation of \cite{BCP}, then the reduced positive root system with multiplicity of $\mathfrak{osp}(2m+1|2n)^{(1)}$, denoted by $\widetilde{\underline{\Delta}}_+$, is given by
\begin{gather*}
\widetilde{\underline{\Delta}}_+=\widehat{\underline{\Delta}}^{\text{re}}_{>}\cup \widetilde{\Delta}^{\text{im}}\cup \widehat{\underline{\Delta}}^{\text{re}}_{<},
\end{gather*}
where $\widehat{\underline{\Delta}}^{\text{re}}_{>}=\{\,\alpha+k\delta\,|\,\alpha\in \underline{\Delta}_+,\ k\geqslant 0\,\}$, $\widehat{\Delta}^{\text{re}}_{<}=\{\,-\alpha+k\delta\,|\,\alpha\in \underline{\Delta}_+,\ k\geqslant 1\,\}$.
$\widetilde{\Delta}^{\text{im}}=\{\,k\delta\,|\,k>0\,\}\times I=\{k\delta^{(i)}|k>0,i\in I\}$, $I=\{1,\ldots,m+n\}$.}

\section{Quantum affine superalgebras for type $\mathfrak{osp}(2m+1|2n)$}
{
Let $q$ be a formal parameter, set
\begin{align*}
 q_i := q^{\frac{|(\alpha_i, \alpha_i)|}{2}}~\textrm{for}~(\alpha_i, \alpha_i)\neq0, \quad q_i := q~\textrm{for}~(\alpha_i, \alpha_i)=0,
\end{align*}where $i=0,1,..,n+m$, and for $ a \in \mathbb{Z}_+$,
\begin{equation*}
\begin{aligned}
&[a]_i=\frac{q_i^a - q_i^{-a}}{q_i - q_i^{-1}},\quad [a]_i! = [1]_i \cdots [a-1]_i [a]_i.
\end{aligned}
\end{equation*}}
For $\mu \in \mathbb{C}(q^{1/2})$, we define
$$[X,\,Y]_\mu = XY - (-1)^{[X][Y]}\mu YX \quad \text{for homogeneous elements } X,Y.$$
Then for $\omega,\nu \in \mathbb{C}(q^{1/2})$, it is straightforward to verify that
\begin{align*}
[[X,\,Y]_\mu,\,Z]_\omega &= [X,\,[Y,\,Z]_\nu]_{\mu\omega\nu^{-1}} + (-1)^{[Y][Z]}\nu[[X,\,Z]_{\omega\nu^{-1}},\,Y]_{\mu\nu^{-1}}, \\
[X,\,[Y,\,Z]_\mu]_\omega &= [[X,\,Y]_\nu,\,Z]_{\mu\omega\nu^{-1}} + (-1)^{[X][Y]}\nu[Y,\,[X,\,Z]_{\omega\nu^{-1}}]_{\mu\nu^{-1}}.
\end{align*}

\subsection{Drinfeld-Jimbo presentation $U_q[\mathfrak{osp}(2m+1|2n)^{(1)}]$}
Recall the Drinfeld-Jimbo presentation of the quantum affine superalgebra associated with $\mathfrak{osp}(2m+1|2n)^{(1)}$ {for the standard parity sequence,} as introduced by H. Yamane (\cite{HYAM}).

\begin{definition}\label{3.1}
The \textit{quantum affine superalgebra} $U_q[\mathfrak{osp}(2m+1|2n)^{(1)}]$  over $\mathbb{C}(q^{1/2})$ is an associative superalgebra of \textit{Chevalley generators}  $\chi_i^{\pm}:=\chi_{\alpha_i}^{\pm}$, $K_i:=K_{\alpha_i}$ $(i=0,1,...,N)$ with the parity of $[\chi_i^{\pm}]=[\alpha_i]$ and $[K_i]=0$, subject to the following relations.
\begin{align}\label{DJ1}
&K_i^{\pm 1}K_i^{\mp 1}=1,\quad K_iK_j=K_jK_i, \\ \label{DJ2}
&K_i\chi_j^{\pm}K_i^{-1}=q_i^{\pm A_{ij}}\chi_j^{\pm}, \\ \label{DJ3}
&[\chi_i^+,\,\chi_j^-]=\delta_{ij}\frac{K_i-K_i^{-1}}{q_i-q_i^{-1}}, \\ \label{DJ4}
&[\chi_i^{\pm},\,\chi_j^{\pm}]=0~~~~\textrm{for}~~A_{ij}=0, \\ \label{DJ5}
&[\![\chi_i^{\pm},[\![\chi_i^{\pm},\,\chi_{i+1}^{\pm}]\!] ]\!]=0~~~~\textrm{for}~~i\neq n,N, \\ \label{DJ6}
&[\![\chi_i^{\pm},[\![\chi_i^{\pm},\,\chi_{i-1}^{\pm}]\!] ]\!]=0~~~~\textrm{for}~~1<i<N,i\neq n, \\ \label{DJ7}
&[\![\chi_i^{\pm},\,[\![\chi_i^{\pm},[\![\chi_i^{\pm},\,\chi_{i-1}^{\pm}]\!] ]\!] ]\!]=0~~~~\textrm{for}~~i=1~~\textrm{or}~~N, \\ \label{DJ8}
&[ [\![ [\![\chi_{n-1}^{\pm},\,\chi_n^{\pm}]\!],\,\chi_{n+1}^{\pm}]\!],\,\chi_n^{\pm}]=0~~~~\textrm{for}~~n>1, \\ \label{DJ9}
&[ [\![ [\![ [\![ [\![ [\![\chi_3^{\pm},\,\chi_2^{\pm}]\!],\,\chi_1^{\pm}]\!],\,\chi_0^{\pm}]\!],\,\chi_1^{\pm}]\!],\,\chi_2^{\pm}]\!],\,\chi_1^{\pm}]=0
~~~~\textrm{for }n=1,m\geq 2, \\ \label{DJ10}
&[\![ [\![\chi_2^{\pm},\,\chi_1^{\pm}]\!],\,[\![ [\![\chi_2^{\pm},\,\chi_1^{\pm}]\!],\,[\![ [\![\chi_2^{\pm},\,\chi_1^{\pm}]\!],\,\chi_0^{\pm}]\!] ]\!] ]\!]\nonumber \\
&=(1-[2]_1)[\![ [\![ [\![\chi_2^{\pm},\,\chi_1^{\pm}]\!],\,[\![\chi_2^{\pm},\,[\![\chi_2^{\pm},\,[\![\chi_1^{\pm},\,\chi_0^{\pm}]\!] ]\!] ]\!] ]\!],\,\chi_1^{\pm}]\!]~~~~\textrm{for }(n,m)=(1,1),
\end{align}
where the notation
\begin{gather*}
 [\![X_{\alpha},\,X_{\beta}]\!]=[X_{\alpha},\,X_{\beta}]_{q^{-(\alpha,\beta)}}=X_{\alpha}X_{\beta}-(-1)^{[\alpha][\beta]}q^{-(\alpha,\beta)}X_{\beta}X_{\alpha}.
\end{gather*}

\end{definition}

Use the notation $K_{\lambda}=\prod_{i=0}^{n+m}\left(K_i\right)^{m_i}$ for $\lambda=\sum_{i=0}^{n+m}m_i\alpha_i$. The elements $K_{\delta}^{\pm\frac{1}{2}}=q^{\pm\frac{1}{2}c}$ are central in $U_q[\mathfrak{osp}(2m+1|2n)^{(1)}]$.

Let $U_q^{+}$ (resp., $U_q^{-}$) be the sub-superalgebra of $U_q[\mathfrak{osp}(2m+1|2n)^{(1)}]$ generated by $\chi_i^{+}$ (resp., $\chi_i^{-}$), and $U_q^{0}$ be the sub-superalgebra generated by $K_i$. {By \cite[Section 6]{HYAM}, we have the triangular decomposition:}
\begin{align*}
 U_q[\mathfrak{osp}(2m+1|2n)^{(1)}]\simeq U_q^{-}\otimes U_q^{0}\otimes U_q^{+}.
\end{align*}{Furthermore, quantum affine superalgebra $U_q(\mathfrak{\hat{g}})$
as a
Hopf superalgebra equipped with the comultiplication $\Delta$, counit $\varepsilon$, and antipode $S$ defined as follows:
\begin{align*}
  &\Delta(\chi_i^{+}) = \chi_i^{+}\otimes 1 + K_i\otimes\chi_i^{+},\quad
  \Delta(\chi_i^{-}) = \chi_i^{-}\otimes K_i^{-1} + 1\otimes\chi_i^{-}, \\
  & \Delta(K_i) = K_i\otimes K_i,\quad \varepsilon(\chi_i^{\pm}) = 0,\quad \varepsilon(K_i^{\pm}) = 1,\\
  &S(\chi_i^{+}) = K_i^{-1}\chi_i^{+},\quad S(\chi_i^{-}) = -\chi_i^{+}K_i, \quad S(K_i) = K_i^{-1}.
\end{align*}}Therefore, we can extend
the results in \cite[Section 1.2.13 \& Proposition 3.1.6]{Lusz2} to the super case as below:
\begin{lemma}\label{ri}
There exist two classes of linear maps ${}_ir$ (resp. $r_i$): $U_q^+ \rightarrow U_q^+$ for $i=0,1,\ldots,N$ such that
\begin{gather*}
{}_ir(1)=0, \quad {}_ir(\chi_j^+)=\delta_{ij} \quad \text{(resp. $r_i(1)=0$, $r_i(\chi_j^+)=\delta_{ij}$)}
\end{gather*}
and for homogeneous elements $X_1,X_2 \in U_q^+$ with $\mathbb{Z}_2$-gradings $[\beta_1],[\beta_2]$,
\begin{align*}
&{}_ir\left(X_1X_2\right) = (-1)^{[\beta_2][\alpha_i]}{}_ir\left(X_1\right)X_2
+ q^{-(\beta_1,\alpha_i)}X_1\,{}_ir\left(X_2\right), \\
&r_i\left(X_1X_2\right) = (-1)^{[\beta_2][\alpha_i]}q^{-(\beta_2,\alpha_i)}r_i\left(X_1\right)X_2
+ X_1\,r_i\left(X_2\right).
\end{align*}
Moreover, $[X,\,\chi_i^-]=\displaystyle\frac{r_i(X)K_i-K_i^{-1}\,{}_ir(X)}{q_i-q_i^{-1}}$ for $X \in U_q^+$ and $i=0,1,\ldots,N$.
\end{lemma}

The next lemma states a well-known result.

\begin{lemma}\cite[Proposition 6.5.1]{HYAM}\label{useful}
    Let $X\in U_q^+$. If $[X, \chi_i^-]=0$ holds for $i=0,1,\ldots,N$, then $X$ must be zero.
\end{lemma}

\begin{proposition}\label{anti-auto}
There exists a superalgebraic anti-automorphism $\Omega$ of $U_q[\mathfrak{osp}(2m+1|2n)^{(1)}]$, defined by
\begin{gather*}
  \chi_i^+ \mapsto d_{i+1}\chi_i^-,\quad \chi_i^- \mapsto d_{i+1}\chi_i^+,\quad K_{\alpha} \mapsto K_{\alpha}^{-1},\quad
  K_{\delta}^{\pm\frac{1}{2}} \mapsto K_{\delta}^{\mp\frac{1}{2}},\quad q \mapsto q^{-1}
\end{gather*}
for $d_i = (\varepsilon_i,\varepsilon_i)$, {where $i = 1,\ldots,N$, and $d_{N+1} = 1$.}
\end{proposition}

For $\alpha=\sum_im_i\alpha_i\in Q$, we define the support of $\alpha$ by
\begin{gather*}
 \operatorname{supp}\alpha=\{i\in\{1,...,N\}|m_i\neq0\}.
\end{gather*}
And on $\widetilde{\underline{\Delta}}_+=\widehat{\underline{\Delta}}^{\text{re}}_{>}\cup \widetilde{\Delta}^{\text{im}}\cup \widehat{\underline{\Delta}}^{\text{re}}_{<}$, we define an order ``$\prec$" as below:

(O1) $\alpha_0\prec\alpha_1\prec\cdots\prec\alpha_{N}$;

(O2) $\delta^{(N)}\prec 2\delta^{(1)}$, and for $k>0$, $k\delta^{(1)}\prec k\delta^{(2)}\prec \cdots \prec k\delta^{(N)}$;

(O3) $\alpha\prec \gamma \prec \beta$, if $\gamma=\alpha+\beta\in \widetilde{\underline{\Delta}}_+$ and $\alpha\neq \beta$;

(O4) {$\beta\prec \alpha$, if $\alpha,\beta\in \widetilde{\underline{\Delta}}_+\setminus \widetilde{\Delta}^{\text{im}}$,} $\alpha-\beta\notin\widetilde{\underline{\Delta}}_+$ and $\alpha-\beta\in Q_+$.

Now, we introduce the quantum root vectors on $\widetilde{\underline{\Delta}}$ as below:

\begin{definition}\label{q-roots}
We define the \textit{quantum root vectors} $E_{\alpha},F_{\alpha}$ for $\alpha\in\widetilde{\underline{\Delta}}_+$ in $U_q[\mathfrak{osp}(2m+1|2n)^{(1)}]$ according to the following rules ($i=1,\ldots,N$),

\vspace{6pt}
(E1)\ $E_{\alpha_i}:=\chi_i^+$, $F_{\alpha_i}:=d_{i+1}\chi_i^-$, $E_{\delta-\theta}:=\chi_0^+$, $F_{\delta-\theta}:=d_1\chi_0^-$.

(E2)\ If $\beta,\alpha_i+\beta\in \Delta_+\cup(\delta-\Delta_+)$, $\beta\prec\alpha_i$ and $\operatorname{supp}\alpha_i\cap \operatorname{supp}\beta=\varnothing$, we define
\begin{gather*}
  E_{\alpha_i+\beta}:=q^{(\alpha_i,\beta)}\frac{q_i-q_i^{-1}}{q^{(\alpha_i,\beta)}
  -q^{-(\alpha_i,\beta)}}[\![E_{\alpha_i},\,E_{\beta}]\!],
\end{gather*}

  (E3)\ If $\beta,\alpha_i+\beta\in \Delta_+\cup(\delta-\Delta_+)$, $\beta\prec\alpha_i$ and $\operatorname{supp}\alpha_i\cap \operatorname{supp}\beta\neq\varnothing$, we define
  $$E_{\alpha_i+\beta}:=(-1)^{[\alpha_i]}q^{(\alpha_i,\beta)}\frac{q_i-q_i^{-1}}{q^{(\alpha_i,\beta)}
  -q^{-(\alpha_i,\beta)}}[\![E_{\alpha_i},\,E_{\beta}]\!],$$
for $\beta\neq\delta-\alpha_1-...-\alpha_{N}$, and
 \begin{gather*}
  E_{\alpha_{N}+\beta}:=(-1)^{[\alpha_{N}]}q^{(\alpha_{N},\beta-\alpha_{N})}\frac{q_{N}-q_{N}^{-1}}{q^{(\alpha_{N},\beta-\alpha_{N})}
  -q^{-(\alpha_{N},\beta-\alpha_{N})}}[\![E_{\alpha_{N}},\,E_{\beta}]\!],
\end{gather*}for $\beta=\delta-\alpha_1-...-\alpha_{N}$.

  (E4)\ Set
  \begin{gather*}
  \exp\left((q_i-q_i^{-1})\sum E_{\delta^{(i)}}z^{-1}\right):=1+(q_i-q_i^{-1})\sum\widetilde{E}_{\delta^{(i)}}z^{-1},
  \end{gather*}
  where $\widetilde{E}_{\delta^{(i)}}:=q^{-(\alpha_i,\alpha_i)}[\![E_{\alpha_i},\,E_{\delta-\alpha_i}]\!]$. Then, if $\beta=\sum_{s=1}^lk_{j_s}\alpha_{j_s}\in\Delta_+$ with $k_{j_1},\ldots,k_{j_l}\neq 0$, we define
      \begin{gather*}
      E_{r\delta+\beta}=b_{\beta}\,[E_{\delta^{(i)}},\,E_{(r-1)\delta+\beta}],\quad
      E_{(r+1)\delta-\beta}=-b_{\beta}\,[E_{\delta^{(i)}},\,E_{r\delta-\beta}],
      \end{gather*}
      and the index $i$ is determined by
      \begin{equation*}
      i=\begin{cases}
      2, &\text{if }n=1,\beta=\alpha_1, \\
      j_1, &\text{if }n>1, j_1=n+1,1, \\
      j_1-1, &\text{if }n=1\text{ and }j_1>2\text{ or }n>1\text{ and } j_1\neq n+1.
      \end{cases}
      \end{equation*}
      and the coefficient $b_{\beta}$ is determined by
      \begin{gather}\label{eq:EF1}
      [E_{\delta+\beta},\,F_{\delta+\beta}]=\frac{K_{\delta+\beta}-K_{\delta+\beta}^{-1}}{q_{\beta}-q_{\beta}^{-1}}\quad \text{for }q_{\beta}:=q^{|(\beta,\beta)|/2},
      \end{gather}
     except for $\beta \in \Pi$, and notably, $b_{\alpha_j}=[A_{i,j}]_{i}^{-1}$.

    (E5)\ For $r>1$, we define the quantum imaginary root vectors $E_{r\delta^{(i)}}$ by the formal power series
  \begin{gather*}
  \exp\left((q_i-q_i^{-1})\sum_{r>1}E_{r\delta^{(i)}}z^{-r}\right):=1+(q_i-q_i^{-1})\sum_{r>1}\widetilde{E}_{r\delta^{(i)}}z^{-r},
  \end{gather*}
  where $\widetilde{E}_{r\delta^{(i)}}:=q^{-(\alpha_i,\alpha_i)}[\![E_{\alpha_i},\,E_{r\delta-\alpha_i}]\!]$.

  (E6)\ $F_{\alpha}=\Omega(E_{\alpha})$.
\end{definition}
Then $U_q^+$ is generated by $E_{\alpha_i}$($i=1,\ldots,N$) and $E_{\delta-\theta}$. Moreover, all elements $E_{\alpha}$($\alpha\in\widehat{\Delta}_+$) are included in $U_q^+$.

\begin{example}\label{Examp} Consider the low rank case of $\mathfrak{osp}(3|2)$, then the reduced roots are
\begin{gather*}
  \alpha_1\prec\alpha_1+\alpha_2\prec\alpha_1+2\alpha_2\prec\alpha_2.
\end{gather*}According the definition of (E1)-(E3), the quantum root vectors of $U_q[\mathfrak{osp}(3|2)^{(1)}]$ as follows
\begin{align*}
&E_{\alpha_1},\quad E_{\alpha_2},\quad E_{\delta-\theta},\\
  &E_{\alpha_1+\alpha_2}:=q^{(\alpha_2,\alpha_1)}\frac{q_2-q_2^{-1}}{q^{(\alpha_2,\alpha_1)}-
  q^{-(\alpha_2,\alpha_1)}}[\![E_{\alpha_2},\,E_{\alpha_1}]\!], \\
  &E_{\alpha_1+2\alpha_2}:=(-1)^{[\alpha_2]}q^{(\alpha_2,\alpha_1+\alpha_2)}\frac{q_2-q_2^{-1}}{q^{(\alpha_2,\alpha_1+\alpha_2)}-
  q^{-(\alpha_2,\alpha_1+\alpha_2)}}[\![E_{\alpha_2},\,E_{\alpha_1+\alpha_2}]\!],\\
  &E_{\alpha_0+\alpha_1}=E_{\delta-\alpha_1-2\alpha_2}:=q^{(\alpha_1,\alpha_0)}\frac{q_1-q_1^{-1}}{q^{(\alpha_1,\alpha_0)}-
  q^{-(\alpha_1,\alpha_0)}}[\![E_{\alpha_1},\,E_{\delta-\theta}]\!], \\
  &E_{\alpha_0+\alpha_1+\alpha_2}=E_{\delta-\alpha_1-\alpha_2}:=q^{(\alpha_2,\alpha_0+\alpha_1)}\frac{q_2-q_2^{-1}}{q^{(\alpha_2,\alpha_0+\alpha_1)}-
  q^{-(\alpha_2,\alpha_0+\alpha_1)}}[\![E_{\alpha_2},\,E_{\delta-\alpha_1-2\alpha_2}]\!], \\
  &E_{\alpha_0+\alpha_1+2\alpha_2}=E_{\delta-\alpha_1}:=(-1)^{[\alpha_2]}q^{(\alpha_2,\alpha_0+\alpha_1)}\frac{q_2-q_2^{-1}}{q^{(\alpha_2,\alpha_0+\alpha_1)}-
  q^{-(\alpha_2,\alpha_0+\alpha_1)}}[\![E_{\alpha_2},\,E_{\delta-\alpha_1-\alpha_2}]\!], \\
  &E_{\alpha_0+2\alpha_1+\alpha_2}=E_{\delta-\alpha_2}:=(-1)^{[\alpha_1]}q^{(\alpha_1,\alpha_0+\alpha_1+\alpha_2)}\frac{q_1-q_1^{-1}}{q^{(\alpha_1,\alpha_0+\alpha_1+\alpha_2)}-
  q^{-(\alpha_1,\alpha_0+\alpha_1+\alpha_2)}}[\![E_{\alpha_1},\,E_{\delta-\alpha_1-\alpha_2}]\!].
\end{align*}While the other quantum root vectors can be defined by induction from above root vectors via (E4)-(E6).
\end{example}

\subsection{Drinfeld presentation $\mathcal{U}_q[\mathfrak{osp}(2m+1|2n)^{(1)}]$}

To construct the Drinfeld presentation of quantum affine superalgebras for type $\mathfrak{osp}(2m+1|2n)$ ($m\geqslant 1,\ n\geqslant 0$) in standard parity, we present the new superalgebra in terms of current generators as given in Definition \ref{Drin-pre}. A similar definition for any parity in this type is provided by Bezerra, Futorny and Kashuba; for more details, see \cite{Luan}.
\begin{definition}\label{Drin-pre}
The superalgebra $\mathcal{U}_q[\mathfrak{osp}(2m+1|2n)^{(1)}]$ over $\mathbb{C}(q^{1/2})$ is an associative superalgebra generated by the elements $x_{i,s}^{\pm}$, $a_{i,r}$, $k_i^{\pm 1}(i=1,...,N, s\in\mathbb{Z}, r\in\mathbb{Z}\backslash\{0\}$), and the central element $q^{\pm \frac{1}{2}c}$, with the following defining relations. The parity of generators $x_{i,s}^{\pm}$ is denoted by $[x_{i,s}^{\pm}]=[\alpha_i]$, while all other generators have parity 0.
\begin{align}\label{DE1}
&q^{\pm\frac{1}{2}c}q^{\mp\frac{1}{2}c}=k_i^{\pm 1}k_i^{\mp 1}=1, \quad k_ik_j=k_jk_i, \\ \label{DE2}
 &k_ia_{j,r}=a_{j,r}k_i, \quad k_ix_{j,s}^{\pm}k_i^{-1}=q_i^{\pm A_{ij}}x_{j,s}^{\pm}, \\ \label{DE3}
&[a_{i,r},\,a_{j,s}]=\delta_{r,-s}\frac{[r A_{ij}]_{i}}{r}\cdot\frac{q^{r c}-q^{-r c}}{q_j-q_j^{-1}}, \\ \label{DE4}
&[a_{i,r},\,x_{j,s}^{\pm}]=\pm\frac{[r A_{ij}]_{i}}{r}q^{\mp\frac{|r|}{2}c}x_{j,r+s}^{\pm}, \\ \label{DE5}
&[x_{i,s}^+,\,x_{j,l}^-]=\delta_{ij}\frac{q^{\frac{s-l}{2}c}\Phi_{i,s+l}^+-q^{\frac{l-s}{2}c}\Phi_{i,s+l}^-}{q_i-q_i^{-1}}, \\ \label{DE6}
&[\![x_{i,s+1}^{\pm},\,x_{j,l}^{\pm}]\!]+[\![x_{j,l+1}^{\pm},\,x_{i,s}^{\pm}]\!]=0, \\ \label{DE7}
&[x_{i,s}^{\pm},\,x_{j,l}^{\pm}]=0,~~~~\textrm{if}~~A_{ij}=0, \\ \label{DE8}
&\operatorname{Sym}_{s_1,s_2}[\![x_{i,s_1}^{\pm},[\![x_{i,s_2}^{\pm},\,x_{i+p,l}^{\pm}]\!] ]\!]=0~~~~\textrm{for}~~i\neq n,N,\ p=\pm 1, \\ \label{DE9}
&\operatorname{Sym}_{s_1,s_2,s_3}[\![x_{N,s_1}^{\pm},\,[\![x_{N,s_2}^{\pm},[\![x_{N,s_3}^{\pm},\,x_{N-1,l}^{\pm}]\!] ]\!] ]\!]=0, \\ \label{DE10}
&\operatorname{Sym}_{l_1,l_2}[ [\![ [\![x_{n-1,s_1}^{\pm},\,x_{n,l_1}^{\pm}]\!],\,x_{n+1,s_2}^{\pm}]\!],\,x_{n,l_2}^{\pm}]=0
~~~~\textrm{for}~~n>1,
\end{align}
where $\Phi_{i,\pm r}^{\pm}$($r\geqslant 0$) is given by the formal power series
\begin{gather}\label{DE11}
    \sum\limits_{r\geqslant 0}\Phi_{i,\pm r}^{\pm}z^{\mp r}:=k_i^{\pm 1}\exp\left(\pm (q_i-q_i^{-1})\sum\limits_{r>0}a_{i,\pm r}z^{\mp r}\right).
\end{gather}
\end{definition}

Given the aforementioned definition of superalgebras, we hereby state the main theorem as follows.

\begin{theorem}\label{Main1}
There exists an isomorphism $\Psi$ of superalgebras between  $\mathcal{U}_q[\mathfrak{osp}(2m+1|2n)^{(1)}]$ and the Drinfeld-Jimbo presentation $U_q[\mathfrak{osp}(2m+1|2n)^{(1)}]$. This isomorphism is defined by the following map:
\begin{align*}
&q^{\pm\frac{1}{2}c} \mapsto K_{\delta}^{\pm\frac{1}{2}},\quad k_i^{\pm 1} \mapsto K_i^{\pm 1},\quad x_{i,0}^{\pm} \mapsto \chi_i^{\pm}, \\
&x_{i,r}^+ \mapsto E_{r\delta+\alpha_i},\quad x_{i,-r}^+ \mapsto (-d_i)^rd_{i+1}F_{r\delta-\alpha_i}K_{\delta}^{r}K_i^{-1}, \\
&x_{i,r}^- \mapsto K_{\delta}^{-r}K_iE_{r\delta-\alpha_i},\quad x_{i,-r}^- \mapsto (-d_i)^rd_{i+1}F_{r\delta+\alpha_i}, \\
&a_{i,r} \mapsto K_{\delta}^{-\frac{r}{2}}E_{r\delta^{(i)}},\quad a_{i,-r} \mapsto (-d_i)^rK_{\delta}^{\frac{r}{2}}F_{r\delta^{(i)}},
\end{align*}where $i=1,\ldots,N$.
\end{theorem}
We will provide a proof of this theorem in Section 5.

\section{Braid group action on Chevalley generators}

{ In this section, we investigate the braid group action on Drinfeld-Jimbo presentation $U_q[\mathfrak{osp}(2m+1|2n)^{(1)}]$ and present some related results.  Following \cite[Section 4]{Luan}, we will describe the Dynkin diagrams of type $\mathfrak{osp}(2m+1|2n)^{(1)}$. Here, the white node
\begin{tikzpicture}
    \draw (0.22,0) circle (0.22);
\end{tikzpicture}
means the corresponding simple root $\alpha$ is even; the gray node
\begin{tikzpicture}
    \draw (0.22,0) circle (0.22);
    \draw (0.06,0.16)--(0.38,-0.16)  (0.06,-0.16)--(0.38,0.16);
\end{tikzpicture}
means the corresponding simple root $\alpha$ is odd and $(\alpha,\alpha)=0$; while the black node
\begin{tikzpicture}
    \draw[fill=black] (0.22,0) circle (0.22);
\end{tikzpicture}
means the corresponding simple root $\alpha$ is odd and $(\alpha,\alpha)\neq 0$.

\subsection{Odd reflections and Dynkin diagrams}

Let $r_{\alpha}$ be the reflection of $\mathfrak{osp}(2m+1|2n)^{(1)}$ corresponding to a root $\alpha\in \widehat{\Delta}$. In contrast to the even cases, not all simple reflections preserve the Dynkin diagram of $\mathfrak{osp}(2m+1|2n)^{(1)}$. Such unusual reflections are termed \textit{odd reflections}. If $\alpha$ is an odd root such that $(\alpha,\alpha)=0$, $r_{\alpha}$ is an odd reflection which acts on $\widehat{\Delta}$ described as: for $\beta\in \widehat{\Delta}$,
\begin{equation*}
     r_{\alpha}(\beta)=\begin{cases}
         -\alpha,   &\textrm{if}~~\beta=\alpha, \\
         \alpha+\beta,    &\textrm{if}~~(\alpha,\beta)\neq 0, \\
         \beta,   &\textrm{other cases.}
     \end{cases}
\end{equation*}
Otherwise, the action of $r_{\alpha}$ on $\widehat{\Delta}$ is consistent with that in the even case. For simplicity, set $r_i:=r_{\alpha_i}$. Define recursively,
\begin{align*}
    \omega_1&=r_0r_1r_2\cdots r_{N-1}r_N\cdots r_1, \\
    \omega_2&=r_1\omega_1 r_1\omega_1, \\
    \omega_{i+1}&=r_i\omega_ir_i\omega_i\omega_{i-1}^{-1},\quad i=2,\ldots,N-1.
\end{align*}
We find that each $\omega_i$ preserve the Dynkin diagram and $\omega_i(\alpha_j)=\alpha_j-\delta_{ij}\delta$.
\begin{example}
    Consider the affine Lie superalgebra $\mathfrak{osp}(3|2)^{(1)}$. In this case, $N=2$. Figure 1 illustrates the actions of $w_1$ and odd reflections on the Dynkin diagram of $\mathfrak{osp}(3|2)^{(1)}$.
    \begin{center}
    \begin{tikzpicture}
      %  һ  ͼ
      \draw (0.22,0) circle (0.22);
      \node[above] at (0.22,0.4) {$0$};
      \draw[double, double distance=3pt] (0.44,0) -- (1.1,0);
      \draw[-{To[scale=2, line width=0.4pt]}, line width=0.4pt] (1.2,0) -- ++(0.01,0);
      \draw (1.46,0) circle (0.22);
      \node[above] at (1.46,0.4) {$1$};
      \draw (1.30,-0.16) -- (1.62,0.16) (1.3,0.16) -- (1.62,-0.16);
      \draw[double, double distance=3pt] (1.68,0) -- (2.34,0);
      \draw[-{To[scale=2, line width=0.4pt]}, line width=0.4pt] (2.44,0) -- ++(0.01,0);
      \draw (2.70,0) circle (0.22);
      \node[above] at (2.70,0.4) {$2$};
      %  ͷ
      \draw[-To] (3.60,0) -- (5.60,0);
      \node at (4.60,0.2) {$r_1$};
      %%%%%%%%%%%%%%%%%%%%%%%%%%%%%%%%%%%%%%%%%%%%%%%%%%%%%%%%%%%%%%%%%
      % ڶ   ͼ
      \draw (6.84,0.6) circle (0.22) (6.68,0.44) -- (7.00,0.76) (6.68,0.76) -- (7.00,0.44);
      \node at (6.32,0.6) {$0$};
      \draw (6.84,-0.6) circle (0.22) (6.68,-0.44) -- (7.00,-0.76) (6.68,-0.76) -- (7.00,-0.44);
      \node at (6.32,-0.6) {$1$};
      \draw[dash pattern=on 3pt off 2pt]  (6.94,0.38)--(6.94,-0.38) (6.74,0.38)--(6.74,-0.38); %  Զ         ʽ
      \draw[double, double distance=3pt] (7.05,0.51) -- (7.70,0.2);
      %      ͷ    
      \pgfmathsetmacro{\arrowLength}{0.01}
      \pgfmathsetmacro{\angle}{-25.5}
    %      յ     
      \pgfmathsetmacro{\endXone}{7.78 + \arrowLength * cos(\angle)}
      \pgfmathsetmacro{\endYone}{0.16 + \arrowLength * sin(\angle)}
    %     ָ   ض  Ƕȵļ ͷ  ȷ    ʼλ    ˫ֱ   յ һ  
      \draw[-{To[scale=2, line width=0.4pt]}, line width=0.4pt] (7.78,0.16) -- (\endXone,\endYone);
      \draw[double, double distance=3pt] (7.05,-0.51) -- (7.70,-0.2);
      %      ͷ    
      \pgfmathsetmacro{\arrowLength}{0.01}
      \pgfmathsetmacro{\angle}{25.5}
    %      յ     
      \pgfmathsetmacro{\endXtwo}{7.78 + \arrowLength * cos(\angle)}
      \pgfmathsetmacro{\endYtwo}{-0.16 + \arrowLength * sin(\angle)}
    %     ָ   ض  Ƕȵļ ͷ  ȷ    ʼλ    ˫ֱ   յ һ  
      \draw[-{To[scale=2, line width=0.4pt]}, line width=0.4pt] (7.78,-0.16) -- (\endXtwo,\endYtwo);
    \draw[fill=black] (8.02,0) circle (0.22);%    һ  ʵ ĵĺ Բ
    \node at (8.42,0.4) {$2$};
      %  ͷ
      \draw[-To] (7.20,-1.00) -- (6.00,-1.80);
      \node at (6.90,-1.55) {$r_2$};
      %%%%%%%%%%%%%%%%%%%%%%%%%%%%%%%%%%%%%%%%%%%%%%%%%%%%%%%%%%%%%%%%
      %      ͼ
      \draw (4.34,-1.40) circle (0.22) (4.18,-1.56) -- (4.50,-1.24) (4.18,-1.24) -- (4.50,-1.56);
      \node at (3.82,-1.40) {$0$};
      \draw (4.34,-2.6) circle (0.22) (4.18,-2.44) -- (4.50,-2.76) (4.18,-2.76) -- (4.50,-2.44);
      \node at (3.82,-2.6) {$1$};
      \draw[dash pattern=on 3pt off 2pt]  (4.44,-1.62)--(4.44,-2.38) (4.24,-1.62)--(4.24,-2.38); %  Զ         ʽ
      \draw[double, double distance=3pt] (4.55,-1.49) -- (5.20,-1.80);
      %      ͷ    
      \pgfmathsetmacro{\arrowLength}{0.01}
      \pgfmathsetmacro{\angle}{-25.5}
    %      յ     
      \pgfmathsetmacro{\endXone}{5.28 + \arrowLength * cos(\angle)}
      \pgfmathsetmacro{\endYone}{-1.84 + \arrowLength * sin(\angle)}
    %     ָ   ض  Ƕȵļ ͷ  ȷ    ʼλ    ˫ֱ   յ һ  
      \draw[-{To[scale=2, line width=0.4pt]}, line width=0.4pt] (5.28,-1.84) -- (\endXone,\endYone);
      \draw[double, double distance=3pt] (4.55,-2.51) -- (5.20,-2.2);
      %      ͷ    
      \pgfmathsetmacro{\arrowLength}{0.01}
      \pgfmathsetmacro{\angle}{25.5}
    %      յ     
      \pgfmathsetmacro{\endXtwo}{5.28 + \arrowLength * cos(\angle)}
      \pgfmathsetmacro{\endYtwo}{-2.16 + \arrowLength * sin(\angle)}
    %     ָ   ض  Ƕȵļ ͷ  ȷ    ʼλ    ˫ֱ   յ һ  
      \draw[-{To[scale=2, line width=0.4pt]}, line width=0.4pt] (5.28,-2.16) -- (\endXtwo,\endYtwo);
    \draw[fill=black] (5.52,-2) circle (0.22);%    һ  ʵ ĵĺ Բ
    \node at (5.92,-2.40) {$2$};
    %  ͷ
      \draw[-To] (3.20,-2.20) -- (2.00,-3.00);
      \node at (2.4,-2.4) {$r_1$};
      %%%%%%%%%%%%%%%%%%%%%%%%%%%%%%%%%%%%%%%%%%%%%%%%%%%%%%%%%%%%%%%%
      %   ĸ ͼ
      \draw (0.22,-4.2) circle (0.22);
      \node[above] at (0.22,-3.80) {$0$};
      \draw[double, double distance=3pt] (0.44,-4.2) -- (1.1,-4.2);
      \draw[-{To[scale=2, line width=0.4pt]}, line width=0.4pt] (1.2,-4.2) -- ++(0.01,0);
      \draw (1.46,-4.2) circle (0.22);
      \node[above] at (1.46,-3.80) {$1$};
      \draw (1.30,-4.36) -- (1.62,-4.04) (1.3,-4.04) -- (1.62,-4.36);
      \draw[double, double distance=3pt] (1.68,-4.2) -- (2.34,-4.2);
      \draw[-{To[scale=2, line width=0.4pt]}, line width=0.4pt] (2.44,-4.2) -- ++(0.01,0);
      \draw (2.70,-4.2) circle (0.22);
      \node[above] at (2.70,-3.80) {$2$};
      %  ͷ
      \draw[-To] (3.60,-4.2) -- (5.60,-4.2);
      \node at (4.60,-4.40) {$r_0$};
      %%%%%%%%%%%%%%%%%%%%%%%%%%%%%%%%%%%%%%%%%%%%%%%%%%%%%%%%%%%%%%%%
      %     ͼ
      \draw (6.52,-4.2) circle (0.22);
      \node[above] at (6.52,-3.80) {$0$};
      \draw[double, double distance=3pt] (6.74,-4.2) -- (7.40,-4.2);
      \draw[-{To[scale=2, line width=0.4pt]}, line width=0.4pt] (7.50,-4.2) -- ++(0.01,0);
      \draw (7.76,-4.2) circle (0.22);
      \node[above] at (7.76,-3.80) {$1$};
      \draw (7.60,-4.36) -- (7.92,-4.04) (7.60,-4.04) -- (7.92,-4.36);
      \draw[double, double distance=3pt] (7.98,-4.2) -- (8.64,-4.2);
      \draw[-{To[scale=2, line width=0.4pt]}, line width=0.4pt] (8.74,-4.2) -- ++(0.01,0);
      \draw (9.00,-4.2) circle (0.22);
      \node[above] at (9.00,-3.80) {$2$};
    \end{tikzpicture}\\
    \vspace{0.5em} 
    \textbf{Figure 1}
    \end{center}

\end{example}

\subsection{Braid group for type $\mathfrak{osp}(2m+1|2n)^{(1)}$}
\begin{definition}
The \textit{braid group} $\mathfrak{B}_{N}$ is generated by $T_i\ (i=0,...,N)$  subject to the following relations:
\begin{align*}
&T_iT_j=T_jT_i,\quad  j\neq i\pm 1, \\
&T_iT_{i+1}T_i=T_{i+1}T_iT_{i+1},\quad i\neq1,N-1, \\
&T_iT_{i+1}T_iT_{i+1}=T_{i+1}T_iT_{i+1}T_i,\quad i=1,N-1.
\end{align*}
\end{definition}
Within Yamane's framework (\cite{HYAM}), the braid generators can act as a series of superalgebraic isomorphisms on quantum affine superalgebras for type $\mathfrak{osp}(2m+1|2n)$. Following \cite[Proposition 7.41 \& 7.42]{HYAM}, we have
\begin{proposition}\label{Braid_Ya}
    Fix $i\in\{0,...,N\}$. Let $^{r_i}\mathfrak{osp}(2m+1|2n)^{(1)}$ denote the Lie superalgebra associated with the Dynkin diagram deduced from the action of $r_i$. There exists an isomorphism $T_i$ of superalgebras from $U_q[\mathfrak{osp}(2m+1|2n)^{(1)}]$ to $U_q[^{r_i}\mathfrak{osp}(2m+1|2n)^{(1)}]$ given by
    \begin{align*}
        &T_i(\chi_i^+)=\xi_i^+\chi_i^-K_i,\quad T_i(\chi_i^-)=\xi_i^-K_i^{-1}\chi_i^+,\quad T_i(K_{\alpha})=K_{r_i(\alpha)}, \\
        &T_i(\chi_j^+)=\xi_j^+[\![\ldots[\![\chi_j^+,\,\underbrace{\chi_i^+]\!],\ldots,\chi_i^+ ]\!]}_{s~\text{times}~\chi_i^+},\quad T_i(\chi_j^-)=\xi_j^-[\![\ldots[\![\chi_j^-,\,\underbrace{\chi_i^-]\!],\ldots,\chi_i^- ]\!]}_{s~\text{times}~\chi_i^-},
    \end{align*}
    if $i\neq j,\ r_i(\alpha_j)=\alpha_j+s\alpha_i$. Here, the constants $\xi_i^{\pm},\xi_j^{\pm}$ depend on $i,j$ and Dynkin diagram of $\mathfrak{osp}(2m+1|2n)^{(1)}$.
\end{proposition}The explicit action of $T_i$ was provided by Luan-Foturny-Kashuba in their work \cite{Luan}. In Appendix A, we have listed some of these actions that are relevant to our study.
\begin{remark}
   As stated in \cite{Luan,HYAM}, the operator $T_i$ generally fails to be an automorphism of the quantum superalgebra $U_q[\mathfrak{osp}(2m+1|2n)^{(1)}]$, since the reflection $r_i$ it depends on may alter the Dynkin diagram.
\end{remark}
Clearly, there exists a one-to-one correspondence between the braid generator $T_i$ and the reflection $r_i$.
Define the elements of $\mathfrak{B}_{n+m}$ as below
\begin{align*}
&T_{\omega_1}=T_0T_1\cdots T_NT_{N-1}\cdots T_1, \\
&T_{\omega_2}=T_1^{-1}T_{\omega_1}T_1^{-1}T_{\omega_1}, \\
&T_{\omega_{i+1}}=T_{\omega_{i-1}}^{-1}T_{\omega_i}T_i^{-1}T_{\omega_i}T_i^{-1},\quad i=2,\ldots,N-1.
\end{align*}
Then for $j\neq i$, we have
$$ T_{\omega_i}T_{\omega_j}=T_{\omega_j}T_{\omega_i},\quad T_{\omega_i}(K_j)=K_{\omega_i(\alpha_j)},\quad T_{\omega_i}(\chi_{j}^{\pm})=\chi_j^{\pm}.$$
Each element $ T_{\omega_i}$ forms an automorphism of superalgebra
$U_q[\mathfrak{osp}(2m+1|2n)^{(1)}]$.

In the subsequent part of this section, we will establish several essential lemmas to show the connection between the action of all $T_{\omega_i}$ on Chevalley generators and quantum root vectors we defined in Section 3.1.

\subsection{Connection among braid elements and quantum root vectors}
As a directly conclusion of \cite[Lemma 5.5]{Luan}, we obtain
\begin{lemma}\label{LM:T1}
In $U_q[\mathfrak{osp}(2m+1|2n)^{(1)}]$, the following relations hold for $i=1,\ldots,N$:
\begin{align}\label{useful:1}
&[\chi_i^-,\,K_iT_{\omega_i}(\chi_i^-)]=0, \\ \label{useful:2}
&[\![T_{\omega_i}(\chi_i^-),\,\chi_{i+1}^-]\!]=[\![T_{\omega_{i+1}}(\chi_{i+1}^-),\,\chi_i^-]\!], \\ \label{useful:3}
&[\![T_{\omega_i}^{-1}(\chi_i^+),\,K_i^{-1}\chi_i^-]\!]=[\![\chi_i^+,\,T_{\omega_i}(K_i^{-1}\chi_i^-)]\!].
\end{align}
\end{lemma}

Now we need to formulate the commutation relations of quantum root vectors $E_{\alpha}$ for $\alpha\in\widetilde{\underline{\Delta}}_+$.
\begin{lemma}\label{LM:T2}
In $U_q[\mathfrak{osp}(2m+1|2n)^{(1)}]$, the following relations hold for $i=1,\ldots,N$:
\begin{align}\label{eq:KE}
&K_iE_{\delta-\alpha_i}K_i^{-1}=q^{-(\alpha_i,\alpha_i)}E_{\delta-\alpha_i},\quad K_iF_{\delta-\alpha_i}K_i^{-1}=q^{(\alpha_i,\alpha_i)}F_{\delta-\alpha_i}, \\ \label{eq:EF2}
&[E_{\delta-\alpha_i},\,F_{\delta-\alpha_i}]=-(-1)^{[\alpha_i]}\frac{K_{\delta-\alpha_i}-K_{\delta-\alpha_i}^{-1}}{q_i-q_i^{-1}}.
\end{align}
\end{lemma}
\begin{proof}
All the relations in this lemma can be directly verified by the Definition \ref{q-roots}.
\end{proof}

Next, we simply set
\begin{align*}
&[\![X_1,\ldots,X_t]\!]_{\ell}=[\![X_1,\,[\![X_2,\ldots,\, [\![X_{t-1},\,X_t]\!]\ldots ]\!], \\
&[\![X_1,\ldots,X_t]\!]_r=[\![\ldots [\![X_1,\,X_2]\!],\,\ldots,\,X_{t-1}]\!],\,X_t ]\!].
\end{align*}
\begin{lemma}\label{LM:T3}
In $U_q[\mathfrak{osp}(2m+1|2n)^{(1)}]$, the following relations hold for $i=1,\ldots,N$:
\begin{align}\label{useful:4}
&[E_{\delta-\alpha_i},\,F_{\alpha_i}]=0,\quad [F_{\delta-\alpha_i},\,E_{\alpha_i}]=0, \\ \label{useful:5}
&d_{i+1}K_i\,[E_{\delta-\alpha_i},\,F_{\alpha_{i+1}}]+d_{i+2}K_{i+1}\,[E_{\delta-\alpha_{i+1}},\,F_{\alpha_i}]=0, \\ \label{useful:6}
&d_{i+1}[F_{\delta-\alpha_i},\,E_{\alpha_{i+1}}]K_i^{-1}+d_{i+2}[F_{\delta-\alpha_{i+1}},\,E_{\alpha_i}]K_{i+1}^{-1}=0.
\end{align}
\end{lemma}

\begin{proof}We present the proof for the low rank case with $m=n=1$, since the general case can be derived from it by induction.

For relations \eqref{useful:4}$_1$ with $i=1$, by Lemma \ref{ri} and the definition (E1), the left-hand side is equal to
\begin{gather*}
  [E_{\delta-\alpha_1},\,F_{\alpha_1}]=\displaystyle\frac{r_1(E_{\delta-\alpha_1})K_1-K_1^{-1}\,{}_1r(E_{\delta-\alpha_1})}{q_1-q_1^{-1}},
\end{gather*}
for $c_1\in\mathbb{C}(q^{1/2})$, we find that
\begin{align*}
  r_1(E_{\delta-\alpha_1})&=c_1[\![E_{\alpha_2}, E_{\alpha_2}, E_{\alpha_1}, E_{\alpha_0}]\!]_{\ell}\\
  &=c_1\big(E_{\alpha_2}r_1([\![E_{\alpha_2}, E_{\alpha_1}, E_{\alpha_0}]\!]_{\ell})-r_1([\![E_{\alpha_2}, E_{\alpha_1}, E_{\alpha_0}]\!]_{\ell})E_{\alpha_2}\big)\\
  &=c_1\big(E_{\alpha_2}^2r_1([\![ E_{\alpha_1}, E_{\alpha_0}]\!])+q^2r_1([\![ E_{\alpha_1}, E_{\alpha_0}]\!])E_{\alpha_2}^2-(1+q^2)E_{\alpha_2}r_1([\![ E_{\alpha_1}, E_{\alpha_0}]\!])E_{\alpha_2}\big)\\
  &=0,
\end{align*}since $r_1([\![ E_{\alpha_1}, E_{\alpha_0}]\!])=r_1([ E_{\alpha_1}, E_{\alpha_0}]_{q^{-2}})=r_1([ \chi_{1}^+, \chi_{0}^+]_{q^{-2}})=0$. Similarly, we can show that  ${}_1r(E_{\delta-\alpha_1})=0$, thus
\begin{gather*}
  [E_{\delta-\alpha_1},\,F_{\alpha_1}]=0.
\end{gather*}Likewise,
\begin{gather*}
  [E_{\delta-\alpha_2},\,F_{\alpha_2}]=0.
\end{gather*}

A similar calculation can verify relation \eqref{useful:5}, we omit this process here.
The remaining relations can be deduced by applying the operator $\Omega$ to the former.

\end{proof}

\begin{lemma}\label{LM:T4}
For some $c_i\in\mathbb{C}(q^{1/2})$, the following relations hold for for $i\in\{1,...,N\}\backslash\{n\}$:
\begin{align}\label{useful:7}
&T_{\omega_{i}}(\chi_i^-) = c_iK_{\delta}^{-1}K_iE_{\delta-\alpha_i}, \\ \label{useful:8}
&T_{\omega_{i}}^{-1}(\chi_i^+) = c_iE_{\delta+\alpha_i} .
\end{align}
\end{lemma}

\begin{proof}{
For any  $a,b\in U_q[^{\omega}\!\mathfrak{osp}(2m+1|2n)^{(1)}]$ with $\omega\in \widetilde{W}$ we will write $a\sim b$ if there exists
$c\in\mathbb{C}(q^{1/2})$
such that $a=cb$.}
We begin by establishing relation \eqref{useful:7}.
According to Proposition \ref{Braid_Ya} and Appendix A, we deduce that $T_{\omega_1}(\chi_1^-)\sim K_{\delta}^{-1}K_1E_{\delta-\alpha_1}$ by immediately calculations.

Now, for $i=2$ with $[\alpha_1]=1$, we observe:
\begin{align*}
T_{\omega_2}(\chi_2^-) &\sim K_{\delta}^{-1}K_2[\![\chi_3^+,\ldots,\chi_{N}^+,\chi_{N}^+,\ldots,\chi_3^+,[\![\chi_1^+,\,\chi_2^+]\!],\chi_1^+,\chi_0^+ ]\!]_{\ell} \\
&\sim K_{\delta}^{-1}K_2[\![\chi_1^+,\chi_3^+,\ldots,\chi_{N}^+,\chi_{N}^+,\ldots,\chi_2^+,\chi_1^+,\chi_0^+ ]\!]_{\ell} \\
&\sim K_{\delta}^{-1}K_2 E_{\delta-\alpha_2},
\end{align*}
where a straightforward calculation yields:
\begin{align*}
[\![ [\![\chi_1^+,\,\chi_2^+]\!],\,[\![\chi_1^+,\,\chi_0^+]\!] ]\!] &=[[\chi_1^+,\,\chi_2^+]_q,\,[\chi_1^+,\,\chi_0^+]_{q^{-2}}]_{q^{-1} } \\
&=[\chi_1^+,\,[\chi_2^+,\,[\chi_1^+,\,\chi_0^+]_{q^{-2}}]_q ]_{q^{-1}}+q[[\chi_1^+,\,[\chi_1^+,\,\chi_0^+]_{q^{-2}}]_{q^{-2}},\,\chi_2^+] \\
&=[\![\chi_1^+,\chi_2^+,\chi_1^+,\chi_0^+]\!]_{\ell}.
\end{align*}
For the case $[\alpha_1]=0$, we claim that $[\chi_1^+,\,K_2E_{\delta-\alpha_2}]=0$. Indeed, one can verify:
\begin{gather*}
[\chi_1^+,\,K_2E_{\delta-\alpha_2}]\sim K_2[\![\chi_3^+,\,\ldots,\,\chi_{N}^+,\,\chi_{N}^+,\,\ldots,\,\chi_1^+,\,\chi_1^+,\,\chi_2^+,\,\chi_1^+,\,\chi_0^+]\!]_{\ell}.
\end{gather*}
It is straightforward to observe that the element $[\![\chi_1^+,\,\chi_1^+,\,\chi_2^+,\,\chi_1^+,\chi_0^+]\!]_{\ell}$ commutes with $\chi_i^-$ for all $i$. This implies $[\![\chi_1^+,\,\chi_1^+,\,\chi_2^+,\,\chi_1^+,\chi_0^+]\!]_{\ell}=0$ by virtue of Lemma \ref{useful}.

By commuting both sides of relation \eqref{useful:5} for $i=1$ with $E_{\alpha_1}$, we obtain:
\begin{align*}
0 = &[E_{\alpha_1},\,d_2K_1[E_{\delta-\alpha_1},\,F_{\alpha_2}]+d_3K_2[E_{\delta-\alpha_2},\,F_{\alpha_1}]\,] \\
= & d_2K_1q^{-(\alpha_1,\alpha_1)}[E_{\alpha_1},\,[E_{\delta-\alpha_1},\,F_{\alpha_2}]]_{q^{(\alpha_1,\alpha_1)}}
+ d_3K_2q^{-(\alpha_1,\alpha_2)}[E_{\alpha_1},\,[E_{\delta-\alpha_2},\,F_{\alpha_1}]]_{q^{(\alpha_1,\alpha_2)}} \\
= & d_2K_1q^{-(\alpha_1,\alpha_1)}\left([[E_{\alpha_1},\,E_{\delta-\alpha_1}]_{q^{(\alpha_1,\alpha_1)}},\,F_{\alpha_2}]
+(-1)^{[\alpha_1]}q^{(\alpha_1,\alpha_1)}[E_{\delta-\alpha_1},\,[E_{\alpha_1},\,F_{\alpha_2}]]_{q^{-(\alpha_1,\alpha_1)}}\right) \\
& + d_3K_2q^{-(\alpha_1,\alpha_2)}\left([[E_{\alpha_1},\,E_{\delta-\alpha_2}]_{q^{(\alpha_1,\alpha_2)}},\,F_{\alpha_1}]
+q^{(\alpha_1,\alpha_2)}[E_{\delta-\alpha_2},\,[E_{\alpha_1},\,F_{\alpha_1}]]_{q^{-(\alpha_1,\alpha_2)}}\right) \\
= & d_2K_1[E_{\delta^{(1)}},\,F_{\alpha_2}]+d_2d_3K_2[A_{12}]_1K_1E_{\delta-\alpha_2}.
\end{align*}

This yields:
\begin{gather*}
[E_{\delta^{(1)}},\,\chi_2^-]=-[A_{12}]_1K_2E_{\delta-\alpha_2}.
\end{gather*}

Similarly, by commuting relation \eqref{useful:2} for $i=1$ on both sides with $\chi_1^+$, one has:
\begin{gather*}
T_{\omega_2}(\chi_2^-)\sim [E_{\delta^{(1)}},\,\chi_2^-]\sim K_2E_{\delta-\alpha_2},
\end{gather*}
since $[\chi_1^+,\,T_{\omega_2}(\chi_2^-)]=T_{\omega_2}[\chi_1^+,\,\chi_2^-]=0$.

Now, we will prove it for $i > 2$. Suppose it holds for $i = p$. Then we have
\begin{align*}
T_{\omega_{p+1}}(\chi_{p+1}^-) &= T_{\omega_{p-1}}^{-1}T_{\omega_p}T_{\omega_{p+1}}(\chi_{p+1}^-) = T_p^{-1}T_{\omega_p}T_p^{-1}(\chi_{p+1}^-) \\
&\sim T_p^{-1}T_{\omega_p}([\![\chi_p^-,\,\chi_{p+1}^-]\!]) \\
&\sim T_p^{-1}([\![K_pE_{\delta-\alpha_p},\,\chi_{p+1}^-]\!]) \\
&\sim T_p^{-1}([\![\chi_1^+,\,\ldots,\,\chi_{p-1}^+,\,\chi_{p+2}^+,\,\ldots,\,\chi_{N}^+,\,\chi_{N}^+,\,\ldots,\chi_1^+,\chi_0^+]\!]_{\ell}) \\
&\sim [\![\chi_1^+,\,\ldots,\,\chi_{p-2}^+,\,[\![\chi_p^+,\,\chi_{p-1}^+]\!],\,\chi_{p+2}^+,\,\ldots,\,\chi_{N}^+,\,
\chi_{N}^+,\,\ldots,\,,\chi_1^+,\chi_0^+]\!]_{\ell} \\
&\sim [\![\chi_1^+,\,\ldots,\,\chi_p^+,\,\chi_{p+2}^+,\,\ldots,\,\chi_{N}^+,\,\chi_{N}^+,\,\ldots,\chi_1^+,\chi_0^+]\!]_{\ell}\sim K_{p+1}E_{\delta-\alpha_{p+1}},
\end{align*}
where we have used
\begin{align*}
[\![\chi_{p-1}^+,\,\chi_p^+,\,\chi_{p-1}^+,\,\ldots,\,\chi_0^+]\!]_{\ell}
&=[\chi_{p-1}^+,\,[\![\chi_p^+,\,\chi_{p-1}^+,\,\ldots,\,\chi_0^+ ]\!]_r ] \\
&=[\![ [\chi_{p-1}^+,\, [\chi_p^+,\,\chi_{p-1}^+]_{q^{d_i}}]_{q^{d_{i-1}}},\,\chi_{p-2}^+,\,\ldots,\,\chi_0^+ ]\!]_r=0.
\end{align*}

Next, we check relation \eqref{useful:8}.
%Relation \eqref{useful:1} is equal to $E_{\delta^{(i)}}=T_{\omega_i}(K_iE_{\delta^{(i)}})$.
By relation \eqref{useful:7}, we have for $i\neq n$,
\begin{gather*}
T_{\omega_i}([\![\chi_i^-,\,T_{\omega_i}^{-1}(\chi_i^-)]\!])=[\![T_{\omega_i}(\chi_i^-),\,\chi_i^-]\!]\sim K_i[E_{\delta-\alpha_i},\,F_{\alpha_i}]=0.
\end{gather*}
Then $[\![T_{\omega_i}^{-1}(\chi_i^+),\,\chi_i^+]\!]=0$ by $\Omega$. By commuting it with $\chi_i^-$ and applying relation \eqref{useful:3}, one promptly obtains

\begin{gather*}
T_{\omega_i}^{-1}(\chi_i^+)=c_i[A_{ii}]_i^{-1}[E_{\delta^{(i)}},\,\chi_i^+]= c_iE_{\delta+\alpha_i}.
\end{gather*}
\end{proof}

\begin{lemma}\label{LM:T6}
It holds for  $i=1,...,N$ in $U_q[\mathfrak{osp}(2m+1|2n)^{(1)}]$,
\begin{gather}\label{useful:9}
[\widetilde{E}_{\delta^{(i)}},\,\widetilde{E}_{2\delta^{(i)}}]=[E_{\delta^{(i)}},\,E_{2\delta^{(i)}}]=0.
\end{gather}
\end{lemma}

\begin{proof}
If $A_{ii}\neq0$, then by Lemma \ref{LM:T4} and the defining relations of (E1) and (E4), we have
 \begin{align*}
   \widetilde{E}_{\delta^{(i)}}=E_{\delta^{(i)}}&=q^{-(\alpha_i,\alpha_i)}[\![E_{\alpha_i},\,E_{\delta-\alpha_i}]\!] \\
   &=q^{-(\alpha_i,\alpha_i)}T_{\omega_i}\big([\![E_{\delta+\alpha_i},\,K_i^{-1}F_{\alpha_i}]\!]\big)\\
   &=T_{\omega_i}\big(-[A_{ii}]_i^{-1}K_i^{-1}[[E_{\delta^{(i)}}, F_{\alpha_i}], E_{\alpha_i}]\big) \\
   &=T_{\omega_i}(E_{\delta^{(i)}}).
 \end{align*}Applying $T_{\omega_{i}}$ to
\begin{gather*}
[\widetilde{E}_{\delta^{(i)}},\,\chi_i^-]=[E_{\delta^{(i)}},\,F_{\alpha_i}]=-[A_{ii}]_{i}K_{i}E_{\delta-\alpha_{i}},
\end{gather*}
we obtain
\begin{gather*}
T_{\omega_{i}}^2(\chi_{i}^-)=T_{\omega_{i}}^2(F_{\alpha_i})=-c_{i}[A_{ii}]_{i}^{-1} [\widetilde{E}_{\delta^{(i)}},\,T_{\omega_{i}}(F_{\alpha_i})]
=c_{i}K_{\delta}^{-1}K_{i}E_{2\delta-\alpha_{i}}.
\end{gather*}
This implies that
\begin{gather*}
T_{\omega_{i}}^{-1}\left(\widetilde{E}_{2\delta^{(i)}}\right)
=c_{i}K_{\delta}^{-1}K_{i}E_{2\delta-\alpha_{i}}
=K_{\delta}^{-1}[T_{\omega_{i}}^{-1}(E_{\alpha_i}),\,T_{\omega_{i}}(F_{\alpha_i})]
=[E_{\alpha_i},\,T_{\omega_{i}}^2(F_{\alpha_i})]=\widetilde{E}_{2\delta^{(i)}},
\end{gather*}where we have used Lemma \ref{LM:T1}.

Now, as the same calculation as outlined in \cite[Proposition 5.1]{LYZ}, one has
 \begin{align*}
   [\widetilde{E}_{2\delta^{(i)}},\widetilde{E}_{\delta^{(i)}}]=q^{-(\alpha_i,\alpha_i)}[A_{ii}]_i^{-1}
   [[\![E_{\alpha_i},[E_{\delta^{(i)}},E_{\delta-\alpha_i}]]\!],E_{\delta^{(i)}}]=[2A_{ii}]_i(q_i^{-1}-q_i)[\widetilde{E}_{\delta^{(i)}},\widetilde{E}_{2\delta^{(i)}}].
 \end{align*}This implies that
\begin{gather*}
[\widetilde{E}_{\delta^{(i)}},\,\widetilde{E}_{2\delta^{(i)}}]=[E_{\delta^{(i)}},\,E_{2\delta^{(i)}}]=0,
\end{gather*}

If $A_{ii}=0$, then we can apply the same arguments as above to derive the relation \eqref{useful:9}. Note that the defining relations in (E4) have been utilized.

\end{proof}

\begin{lemma}\label{LM:T5}
For $i,j=1,\ldots,N$ with $i\neq j$ in $U_q[\mathfrak{osp}(2m+1|2n)^{(1)}]$, the relation
\begin{gather}\label{useful:10}
[E_{\alpha_i},\,K_jE_{\delta-\alpha_j}]=[F_{\alpha_i},\,F_{\delta-\alpha_j}K_j^{-1}]=0
\end{gather}
holds.
\end{lemma}

\begin{proof}By Lemma \ref{LM:T4}, we have
\begin{align*}
 [E_{\alpha_i},\,K_jE_{\delta-\alpha_j}]
 =T_{\omega_j}[T_{\omega_j}^{-1}(\chi_i^+), T_{\omega_j}^{-1}(K_jE_{\delta-\alpha_j})]
 =c_j^{-1}K_{\delta}T_{\omega_j}[\chi_i^+, \chi_j^-]
 =0.
\end{align*}Applying the operator $\Omega$, it is easy to see that
$[F_{\alpha_i},\,F_{\delta-\alpha_j}K_j^{-1}]=0$.
\end{proof}

At the end of this section, we need to emphasize that we used the braid group action to provide some relations among the quantum root vectors in the above lemmas, which play a crucial role in the proof of the subsequent Theorem \ref{Main1}; c.f, Section 5.1-5.2.

 }

\section{The proof of Theorem \ref{Main1}}

To achieve our objective, we will construct a new superalgebra denoted as $\mathcal{U}'$, which is, in fact, isomorphic to two presentations of the quantum affine superalgebra $\mathfrak{osp}(2m+1|2n)^{(1)}$ simultaneously. The introduction of the superalgebra $\mathcal{U}'$ proceeds as follows. Without ambiguity, we use the same notations as the current generators of $\mathcal{U}_q[\mathfrak{osp}(2m+1|2n)^{(1)}]$.

\begin{definition}
The superalgebra $\mathcal{U}'$ is an $\mathbb{C}(q^{1/2})$-associative superalgebra generated by $k_i^{\pm}$, $x_{i,0}^{\pm}$, $x_{i,1}^{-}$, $x_{i,-1}^+$ for $i=1,\ldots,N$ and the central element $q^{\pm\frac{1}{2}c}$, satisfying the following relations,
\begin{align}\label{S1}
&q^{\pm\frac{c}{2}}q^{\mp\frac{c}{2}}=k_i^{\pm 1}k_i^{\mp 1}=1,\quad k_ik_j=k_jk_i,\\ \label{S2}
&k_ix_{j,0}^{\pm}k_i^{-1}=q_i^{\pm A_{ij}}x_{j,0}^{\pm},\quad
k_ix_{j,\pm1}^{\mp}k_i^{-1}=q_i^{\mp A_{ij}}x_{j,\pm1}^{\mp}, \\ \label{S3}
&[\,x_{i,0}^+,\,x_{j,0}^-\,]=\delta_{ij}\frac{k_i-k_i^{-1}}{q_i-q_i^{-1}},\quad
[\,x_{i,-1}^+,\,x_{j,1}^-\,]=\delta_{ij}\frac{q^{-c}k_i-q^c k_i^{-1}}{q_i-q_i^{-1}},\\ \label{S4}
&[\![x_{i,\mp1}^{\pm},\,x_{j,0}^{\pm}]\!]
+[\![x_{j,\mp1}^{\pm},\,x_{i,0}^{\pm}]\!]=0,\\ \label{S5}
&[x_{i,0}^{\pm},\,x_{j,0}^{\pm}]=[x_{i,0}^{\pm},\,x_{j,\mp1}^{\pm}]
=[x_{i,\mp1}^{\pm},\,x_{j,\mp1}^{\pm}]=0,~~~~\textrm{if}~~~~A_{ij}=0, \\ \label{S6}
&[\![x_{i,0}^{\pm},[\![x_{i,0}^{\pm},\,x_{i+\ell,0}^{\pm}]\!] ]\!]=0~~~~\textrm{for}~~i\neq n,N,\ \ell=\pm 1, \\ \label{S7}
&[\![x_{N,0}^{\pm},\,[\![x_{N,0}^{\pm},[\![x_{N,0}^{\pm},\,x_{N-1,0}^{\pm}]\!] ]\!] ]\!]=0, \\ \label{S8}
&[ [\![ [\![x_{n-1,0}^{\pm},\,x_{n,0}^{\pm}]\!],\,x_{n+1,0}^{\pm}]\!],\,x_{n,0}^{\pm}]=0
~~~~\textrm{for}~~n>1, \\ \label{S9}
&\Big[[x_{i,0}^{\pm},\,x_{i,\pm1}^{\mp}],\,\big[[[x_{i,0}^{\pm},\,x_{i,\pm1}^{-\mp}],\,x_{i,0}^{\pm}] _{q_i^{A_{ii}}},\,x_{i,\pm1}^{\mp}\big]_{q_i^{-A_{ii}}}\Big]=0~~~~\text{for }i\neq n.
\end{align}
\end{definition}

It is evident that $\mathcal{U}'$ is generated by a finite set of generators and relations, resulting in significantly fewer elements than the Drinfeld presentation from which it originates, where relation \eqref{S9} seems that new.

We introduce in $\mathcal{U}'$ the parameter $\epsilon=\pm$ or $\pm 1$,
\begin{align}\label{S10}
&\Phi_{i,0}^{\pm}=k_i^{\pm 1}, \\ \label{S11}
&\Phi_{i,1}^+=(q_i-q_i^{-1})q^{\frac{1}{2}c}[x_{i,0}^+,\,x_{i,1}^{-}],\quad \Phi_{i,-1}^-=-(q_i-q_i^{-1})q^{-\frac{ 1}{2}c}[x_{i,-1}^+,\,x_{i,0}^-],  \\ \label{S12}
&x_{j,\epsilon}^{\epsilon}=(q_i^{A_{ij}}-q_i^{-A_{ij}})^{-1}q^{\frac{\epsilon c}{2}}[k_i^{-\epsilon}\Phi_{i,\epsilon}^{\epsilon},\,x_{j,0}^{\epsilon}], \\ \label{S13}
&x_{j,\epsilon(r+1)}^{\pm}=\pm\epsilon(q_i^{A_{ij}}-q_i^{-A_{ij}})^{-1}q^{\pm\frac{1}{2}c}
[k_i^{-\epsilon}\Phi_{i,\epsilon}^{\epsilon},\,x_{j,\epsilon r}^{\pm}], \\\label{S14}
&\Phi_{j,r}^+=(q_j-q_j^{-1})q^{\frac{r}{2}c}[x_{j,0}^+,\,x_{j,r}^{-}],\quad \Phi_{j,-r}^-=-(q_j-q_j^{-1})q^{-\frac{ r}{2}c}[x_{j,-r}^+,\,x_{j,0}^-],
\end{align}	
where $r>0$, $i$ is obtained from $j\in\{1,...,N\}$ via
      \begin{equation*}
      i=\begin{cases}
      2, &\text{if }n=1, j=1, \\
      j, &\text{if }n>1, j=n+1,1, \\
      j-1, &\text{if }n=1\text{ and }j>2\text{ or }n>1\text{ and } j\neq n+1.
      \end{cases}
      \end{equation*}
For example, if $(m,n)=(2,1)$, then the generators $x_{i,\epsilon r}^{\pm}$ for $i=1,2,3$ are defined from $\Phi_{2,\epsilon}^{\epsilon}$ by induction on $r$; if $(m,n)=(1,2)$, then the generators $x_{i,\epsilon r}^{\pm}$ for $i=1,2$ are defined from $\Phi_{1,\epsilon}^{\epsilon}$ by induction on $r$, while $x_{3,\epsilon r}^{\pm}$ are defined from $\Phi_{3,\epsilon}^{\epsilon}$ by induction on $r$.

Given a fixed $i$ such that $1\leqslant i\leqslant N$, we define the elements $a_{i,\pm\ell}\,(\ell\geqslant 1)$ recursively by
\begin{gather*}
   \pm (q_i-q_i^{-1})\sum\limits_{r\geqslant 1}a_{i,\pm r}z^{\mp r}=
   \ln\left(1+\sum\limits_{r\geqslant 1}k_i^{\mp 1}\Phi_{i,\pm r}^{\pm}z^{\mp r}\right).
\end{gather*}

\begin{theorem}\label{the5.1}
Under the construction of the superalgebra $\mathcal{U}'$,
we have
\begin{gather*}
(a)\quad \mathcal{U}_q[\mathfrak{osp}(2m+1|2n)^{(1)}]\simeq \mathcal{U}',\\
(b)\quad U_q[\mathfrak{osp}(2m+1|2n)^{(1)}]\simeq \mathcal{U}'.
\end{gather*}
\end{theorem}

Next, we will engage in preliminary discussions before proving the above theorem.

\vspace{6pt}
{{\bf Discussion of Theorem \ref{the5.1}-(a)}}: we define a map $\Theta_1$ as follows
\begin{gather}
\Theta_1 : \mathcal{U}'\rightarrow\mathcal{U}_q[\mathfrak{osp}(2m+1|2n)^{(1)}]
\end{gather}
with $\Theta_1(X)=X$ for $X\in\mathcal{U}'$. It is evident that $\Theta_1$ qualifies as an algebra homomorphism, as the generators and relations of $\mathcal{U}'$ form a subset of the defining relations of $\mathcal{U}_q[\mathfrak{osp}(2m+1|2n)^{(1)}]$.

By the definition relations, we can also define a map $\Theta_2: \mathcal{U}_q[\mathfrak{osp}(2m+1|2n)^{(1)}]\rightarrow\mathcal{U}'$ such as
\begin{equation}
  x_{i,k}^{\pm}\mapsto x_{i,k}^{\pm},\quad
  a_{i,r}\mapsto a_{i,r},\quad
  k_i\mapsto k_i,\quad
  q^{\pm\frac{1}{2}c}\mapsto q^{\pm\frac{1}{2}c}.
\end{equation}
 It is evident that this map is surjective. Our goal is to demonstrate that $\Theta_2$ qualifies as a homomorphism, and their composites $\Theta_1\circ\Theta_2$ and $\Theta_2\circ\Theta_1$ act as identity maps.

There is a lemma that plays a pivotal role in achieving our desired result. Let us denote
\begin{align*}
 \widetilde{\Phi}_{i,\pm r}^{\pm}=\frac{k_i^{\mp1}}{q_i-q_i^{-1}}\Phi_{i,\pm r}^{\pm}.
\end{align*}
Note that a partition of $r$ is a decreasing sequence of positive integers $r_1\geq r_2\geq...\geq r_k$ such that $r_1+r_2+...+r_k=r$, and let $l(r)=k$. For all possible partitions of $r$, we express $\widetilde{\Phi}_{i,\pm r}^\pm$ as $a_{i,\pm r}+C\sum\limits_{k>1}a_{1,\pm r_1}a_{1,\pm r_2}...a_{1,\pm r_k}$ with $C\in\mathbb{C}(q^{1/2})$.

\begin{lemma}\label{LM5.1}
(1) Suppose \eqref{DE5} holds for $0<\pm(k+l)\leq r$, and either \eqref{DE6} holds for $k=l$ or \eqref{DE7} holds for $k=l-1$ in $\mathcal{U}'$. Then, if $r>0$ and $A_{ij}\neq0$, we obtain
\begin{align*}
&[\widetilde{\Phi}_{i,r}^+,\,x_{j,k}^{\pm}]=q^{\mp\frac{1}{2}c}\big(~q_i^{\pm A_{ij}}\widetilde{\Phi}_{i,r-1}^+ x_{j,k+1}^{\pm}
-q_i^{\mp A_{ij}}x_{j,k+1}^{\pm} \widetilde{\Phi}_{i,r-1}^+~\big),\\
&[\widetilde{\Phi}_{i,-r}^-,\,x_{j,k}^{\pm}]
=q^{\pm\frac{1}{2}c}\big(q_i^{\mp A_{ij}}x_{j,k-1}^{\pm} \widetilde{\Phi}_{i,1-r}^--
q_i^{\pm A_{ij}}\widetilde{\Phi}_{i,1-r}^- x_{j,k-1}^{\pm}\big).
\end{align*}
If $r>0$ and $A_{ij}=0$, then we have
\begin{align*}
&[\widetilde{\Phi}_{i,r}^+,\,x_{j,k}^{\pm}]
=[\widetilde{\Phi}_{i,-r}^-,\,x_{j,k}^{\pm}]
=0.
\end{align*}

(2) Suppose the relations from (1) involving $\widetilde{\Phi}_{i,\pm r}^\pm$ and $x_{j,k}^{\pm}$ hold in $\mathcal{U}'$. If
\begin{align*}
&[\widetilde{\Phi}_{i,\pm r'}^\pm,\,\widetilde{\Phi}_{i,\pm r}^\pm]=0, \quad r', r\in \mathbb{Z}_{>0},
\end{align*}
and
\begin{align*}
&[\widetilde{\Phi}_{i,\pm1}^\pm,\,x_{j,k}^{\pm}]=\pm[A_{ij}]_iq^{\mp\frac{1}{2}c}x_{j,k+1}^{\pm}, \quad k\in\mathbb{Z}.
\end{align*}
Then, for $r\in\mathbb{Z}_{>0}$ and $k\in\mathbb{Z}$, we obtain
\begin{align*}
&[a_{i,\pm r},\,x_{j,k}^{\pm}]=\pm\frac{[rA_{ij}]_i}{r}q^{\mp\frac{r}{2}}x_{j,k+r}^{\pm}.
\end{align*}

\end{lemma}
\begin{proof}
The proof follows a similar structure as presented in \cite[Lemma 2.6, Lemma 2.8, Remark 2.9]{WLZ}, with the substitution of coefficients by $[A_{ij}]_i$.
\end{proof}

\vspace{6pt}
{{\bf Discussion of Theorem \ref{the5.1}-(b)}}: introduce two $\mathbb{C}(q^{1/2})$-linear maps $\Psi_1:\mathcal{U}'\rightarrow U_q[\mathfrak{osp}(2m+1|2n)^{(1)}]$ such that
\begin{align*}
&q^{\pm\frac{1}{2}c}\mapsto K_{\delta}^{\pm\frac{1}{2}},\quad k_i^{\pm 1}\mapsto K_i^{\pm 1},\quad x_{i,0}^{\pm}\mapsto \chi_i^{\pm} \\
&x_{i,1}^-\mapsto K_{\delta}^{-1}K_iE_{\delta-\alpha_i},\quad x_{i,-1}^+\mapsto -(-1)^{[i]}F_{\delta-\alpha_i}K_{\delta}K_i^{-1}
\end{align*}
for $i=1,\ldots,N$, and $\Psi_2:U_q[\mathfrak{osp}(2m+1|2n)^{(1)}]\rightarrow \mathcal{U}'$ such that
\begin{align*}
&K_i^{\pm 1}\mapsto k_i^{\pm 1},\quad\chi_i^{\pm}\mapsto x_{i,0}^{\pm},\quad i=1,\ldots,N, \\
&K_0^{\pm 1}\mapsto (q^ck_1^2k_2^2\cdots k_{N}^2)^{\pm 1},\quad K_{\delta}^{\pm\frac{1}{2}}\mapsto q^{\pm\frac{1}{2}c}, \\
&\chi_0^+\mapsto \nu_0^+\,\Psi_2(K_0)\,[\![x_{1,1}^-,x_{2,0}^-,\ldots,x_{N,0}^-,x_{N,0}^-,\ldots,x_{1,0}^-]\!]_r, \\
&\chi_0^-\mapsto \nu_0^-\,[\![x_{1,-1}^+,x_{2,0}^+,\ldots,x_{N,0}^+,x_{N,0}^+,\ldots,x_{1,0}^+]\!]_r\, \Psi_2(K_0^{-1}) ,
\end{align*}
where $\nu_0^+=-\left([2]_N\right)^{-1}$, $\nu_0^-=(-1)^{[\alpha_1]}\left([2]_N\right)^{-1}q^{2n-2m-(\alpha_1,\alpha_2)}$.

To complete the proof of Theorem \ref{the5.1}-(b), it suffices to demonstrate that $\Psi_1$ and $\Psi_2$ are epimorphisms of superalgebras, and their composites $\Psi_1\circ\Psi_2$ and $\Psi_2\circ\Psi_1$ act as identity maps.

\subsection{ The low rank case of Theorem \ref{the5.1} for $(m,n)=(1,1)$}

{\bf Proof of Theorem \ref{the5.1}-(a):}
First, we show the map $\Theta_2$ is a homomorphism, it suffices to demonstrate the validity of relations \eqref{DE1}-\eqref{DE10} in $\mathcal{U}'$.

\vspace{6pt}
(1)  Relations \eqref{DE1}-\eqref{DE7}

Since $A_{11}=0$, our focus shifts to the indices $i=1,2$ in $\mathcal{U}'$, and we proceed as follows.

({\bf\uppercase{\romannumeral1}}) Relations \eqref{DE1} and \eqref{DE2} are straightforward.

({\bf\uppercase{\romannumeral2}}) The following relation holds:
\begin{align}\label{5B1}
&[a_{2,\epsilon },x_{j,k}^{\pm}]=\pm [A_{2j}]_2q^{\mp\frac{c}{2}}x_{j,k+\epsilon}^{\pm},\quad j=1,2, \quad \text{and} \quad k\in \mathbb{Z}.
\end{align}
By relations \eqref{S4} and \eqref{S5}, we can infer the following:
\begin{align*}
&[a_{2,-\epsilon},x_{j,0}^{\epsilon}]=\epsilon[A_{2j}]_2q^{\mp\frac{c}{2}}x_{j,-\epsilon}^{\epsilon},\quad
[a_{2,\epsilon},x_{j,-\epsilon}^{\epsilon}]=\epsilon[A_{2j}]_2q^{\mp\frac{c}{2}}x_{j,0}^{\epsilon},
\end{align*}
and with relation \eqref{S3}, we obtain
\begin{align*}
&[a_{2,\epsilon},\,a_{2,-\epsilon}]=[A_{22}]_2\frac{q^{\epsilon c }-q^{-\epsilon c}}{q_2-q_2^{-1}}.
\end{align*}
This, combined with the definition and induction on $k$, yields \eqref{5B1}.

({\bf\uppercase{\romannumeral3}}) The following relations are established:
\begin{align}\label{5B2}
&[a_{i,2\epsilon },x_{j,k}^{\pm}]=\pm \frac{[2A_{2j}]_2}{2}q^{\mp c}x_{j,k+2\epsilon}^{\pm},\quad i,j=1,2, \quad \text{and} \quad k\in \mathbb{Z},\\
\label{5B2.1}
&[a_{1,\epsilon },x_{j,k}^{\pm}]=\pm [A_{1j}]_1q^{\mp\frac{c}{2}}x_{j,k+\epsilon}^{\pm},\quad j=1,2, \quad \text{and} \quad k\in \mathbb{Z}.
\end{align}

Concerning relation \eqref{5B2}, we apply $[x_{2,0}^{\mp}, [x_{2,\pm1}^{\mp}, .]]$ to relation \eqref{S7} for $m+n=2$.
This implies in the expression:
\begin{align}\label{5B2.2}
&[\![\,x^{\pm}_{2,\pm1},\,x^{\pm}_{1,0}\,]\!]+[\![\,x^{\pm}_{2,0},\,x^{\pm}_{1,\pm1}\,]\!]=0,
\end{align}
Combining this with relation \eqref{S4}, we derive:
\begin{align}\label{5B3}
&[a_{1,\epsilon },x_{2,0}^{\pm}]=\pm[A_{12}]_1q^{\mp \frac{c}{2}}x_{2,\epsilon}^{\pm}, \quad [a_{1,\epsilon },x_{2,-\epsilon}^{\epsilon}]=\epsilon[A_{12}]_1q^{- \frac{c}{2}\epsilon}x_{2,0}^{\epsilon}.
\end{align}

Furthermore, it is easy to observe that:
\begin{align*}
&a_{1,\epsilon}=\epsilon q^{\frac{\epsilon}{2}c}k_{1}^{-\epsilon}[\,x_{1,0}^\epsilon,\,x_{1,\epsilon}^{-\epsilon}\,]
=\epsilon q^{-\frac{\epsilon}{2}c}k_{1}^{-\epsilon}[\,x_{1,\epsilon}^\epsilon,\,x_{1,0}^{-\epsilon}\,].
\end{align*}
As a consequence, we obtain that:
\begin{equation*}
 \begin{split}
  [a_{1,1},\,x_{2,0}^+]&=q^{-\frac{c}{2}}k_{1}^{-1}[\![\,x_{1,1}^+,\,x_{1,0}^-],\,x_{2,0}^+\,]\!]\\
  &=q^{-\frac{c}{2}}k_{1}^{-1}([[\![\,x_{1,1}^+,\,x_{2,0}^+]\!],\,x_{1,0}^-\,]+
  [\![\,x_{1,1}^+,[\,x_{1,0}^-,\,x_{2,0}^+\,]]\!])\\
  &=-q^{-\frac{c}{2}}k_{1}^{-1}([\![\,x_{2,1}^+,[\,x_{1,0}^+,\,x_{1,0}^-\,]]\!]+
  [\![[\,x_{2,1}^+,\,x_{1,0}^-],\,x_{1,0}^+\,]\!])\\
  &=\frac{-1}{q_1-q_1^{-1}}q^{-\frac{c}{2}}k_{1}^{-1}([\![\,x_{2,1}^+,k_{1}-k_{1}^{-1}]\!]\\
  &=q^{-\frac{c}{2}}[A_{12}]_1x_{2,1}^+,
 \end{split}
\end{equation*} for $[\,x_{2,1}^+,\,x_{1,0}^-]=0$ via \eqref{S6} together with \eqref{5B1}. It holds for the remaining cases.

So again, we apply $[x_{2,0}^{\mp}, [x_{1,\pm1}^{\mp}, .]]$ to relation \eqref{S7},
we get that
\begin{align*}
&[\![\,x^{\pm}_{2,\pm1},\,x^{\pm}_{2,0}\,]\!]=0,
\end{align*}
Combined with relation \eqref{S4}, we obtain:
\begin{align}\label{5B4}
&[a_{2,2\epsilon },x_{2,0}^{\pm}]=\pm\frac{[2A_{22}]_2}{2}q^{\mp c}x_{2,2\epsilon}^{\pm}.
\end{align}

Additionally, note that:
\begin{align}\label{5B4.1}
&a_{2,\pm2}=\frac{k_2^{\mp1}\Phi_{2,\pm2}^\pm}{q_2-q_2^{-1}}-\frac{1}{2}(q_2-q_2^{-1})a_{2,\pm1}^2,
\end{align}
where $\frac{1}{(q_2-q_2^{-1})}\Phi_{2,\pm2}^{\pm}=q^c[x_{2,0}^+,\,x_{2,\pm2}^-]=
q^{-c}[x_{2,\pm2}^+,\,x_{2,0}^-]=[\,x_{2,\pm1}^+,\,x_{2,\pm1}^-\,]$ deduced by $[a_{2,\epsilon},\,a_{2,\epsilon}]=[a_{2,\epsilon},[x_{2,0}^\pm,\,x_{2,\epsilon}^\mp]]=0$. Thus,
\begin{equation*}
 \begin{split}
  [a_{2,2},\,x_{2,0}^+]&=\frac{1}{q_2-q_2^{-1}}[k_2^{-1}\phi_{2,2}^+,\,x_{2,0}^+]-\frac{q_i-q_i^{-1}}{2}[a_{2,1}^2,\,x_{2,0}^+]\\
  &=k_2^{-1}[\![\,x_{2,1}^+,\,x_{2,1}^-],\,x_{2,0}^+\,]\!]
    -\frac{q_2-q_2^{-1}}{2}[A_{22}]_2q^{-\frac{1}{2}2}(x_{2,1}^+a_{2,1}+a_{2,1}x_{2,1}^+)\\
  &=k_1^{-1}[\![\,x_{2,1}^+,\,[x_{2,1}^-,\,x_{2,0}^+]\,]\!]
   -\frac{q_2-q_2^{-1}}{2}[A_{22}]_2q^{-\frac{1}{2}2}(x_{2,1}^+a_{2,1}+a_{2,1}x_{2,1}^+)\\
  &=q^{-\frac{1}{2}c}\left(q_2^{-A_{22}}a_{2,1}x_{2,1}^+-q_2^{A_{22}}x_{2,1}^+a_{2,1}\right)
    -\frac{q_2-q_2^{-1}}{2}[A_{22}]_2q^{-\frac{1}{2}c}(x_{2,1}^+a_{2,1}+a_{2,1}x_{2,1}^+)\\
  &=\frac{[2A_{22}]_2}{2}q^{-c}x_{2,2}^+.
 \end{split}
\end{equation*} similarly for the other cases.

Furthermore, relation \eqref{S9} implies
\begin{align*}
&[a_{2,\pm2},\,a_{2,\pm1}]=0,
\end{align*}
and using equation \eqref{5B4.1}, we can deduce
\begin{align*}
&[a_{2,\pm2},\,a_{2,\mp1}]=0.
\end{align*}
Therefore, induction on $k$ yields \eqref{5B2} for $j=2$.

Similarly, to prove relation \eqref{5B2} for $j=1$, we need to show that
\begin{align}\label{5B5}
&[a_{2,2\epsilon },x_{1,0}^{\pm}]=\pm\frac{[2A_{21}]_2}{2}q^{\mp c}x_{1,2\epsilon}^{\pm}.
\end{align}
For $A_{11}=0$, by relation \eqref{S5}, we have
\begin{align}\label{5B6}
&[a_{1,\epsilon },x_{1,0}^{\pm}]=[a_{1,\epsilon },x_{1,-\epsilon}^{\epsilon}]=0.
\end{align}
Hence, applying $[a_{1,\pm1},.]$ and $[a_{2,\pm1},.]$ to relation \eqref{S4} with $i\neq j$, we obtain
\begin{align*}
 &[\![x_{2,\pm1}^{\pm},\,x_{1,\mp1}^{\pm}]\!]
+[\![x_{2,0}^{\pm},\,x_{1,0}^{\pm}]\!]=[\![x_{1,\pm1}^{\pm},\,x_{2,\mp1}^{\pm}]\!]
+[\![x_{1,0}^{\pm},\,x_{2,0}^{\pm}]\!]=0.
\end{align*}
This implies that
\begin{align*}
&[a_{2,2\epsilon },x_{1,-\epsilon}^{\pm}]=\pm\frac{[2A_{21}]_2}{2}q^{\mp c}x_{1,\epsilon}^{\pm},
\end{align*}
so that
\begin{align*}
[a_{2,2\epsilon },x_{1,0}^{\pm}]&=\pm\frac{1}{[A_{21}]_2}q^{\mp \frac{c}{2}}[a_{2,2\epsilon },[a_{2,\epsilon}, x_{1,-\epsilon}^{\pm}]]\\
&=\pm\frac{1}{[A_{21}]_2}q^{\mp \frac{c}{2}}[a_{2,\epsilon },[a_{2,2\epsilon}, x_{1,-\epsilon}^{\pm}]]\\
&=\pm\frac{[2A_{21}]_2}{2}q^{\mp c}x_{1,2\epsilon}^{\pm},
\end{align*}
which completes \eqref{5B5}.

To verify relation \eqref{5B2.1}, using relations \eqref{5B3}, \eqref{5B6}, and \eqref{5B1}, it suffices to check
\begin{align*}
 &[\,a_{2,\pm1},\,a_{1,\mp1}\,]=[A_{21}]_2\frac{q^{\mp c}-q^{\pm c}}{q_1-q_1^{-1}},\quad [\,a_{2,\pm1},\,a_{1,\pm1}\,]=0.
\end{align*}
The first relation is straightforward by calculation. For the second relation, note that
\begin{align*}
 (q^{\frac{c}{2}}+q^{-\frac{c}{2}})[\,a_{2,1},\,a_{1,1}\,]&=q^{\frac{c}{2}}[\,a_{2,1},\,
 [x_{1,0}^+, x_{1,1}^-\,]+q^{-\frac{c}{2}}[\,a_{2,1},\,
 [x_{1,1}^+, x_{1,0}^-\,]\\
 &=[A_{21}]_2(q^{-c}[x_{1,2}^+, x_{1,0}^-\,]-q^{c}[x_{1,0}^+, x_{1,2}^-\,]).
\end{align*}
Moreover, one has,
\begin{align*}
 0=[\,a_{2,2},\,
 [x_{1,0}^+, x_{1,0}^-\,]]=\frac{[A_{21}]_2}{2}(q^{-c}[x_{1,2}^+, x_{1,0}^-\,]-q^{c}[x_{1,0}^+, x_{1,2}^-\,]).
\end{align*}
Comparing the above two expressions, we conclude that $[\,a_{2,1},\,a_{1,1}\,]=0$, similar to $$[\,a_{2,-1},\,a_{1,-1}\,]=0.$$

Now, in order to complete \eqref{5B2} for $i=1, j=1,2$, it is necessary to demonstrate that
\begin{align}
&[a_{1,2\epsilon },x_{j,0}^{\pm}]=\pm\frac{[2A_{1j}]_1}{2}q^{\mp c}x_{j,2\epsilon}^{\pm}.
\end{align}
This calculation is straightforward, similar to \eqref{5B4}, using relations \eqref{5B2.1} and \eqref{5B2.2}.

\textbf{\uppercase{\romannumeral3}}) For $i\neq j$ and $i, j=1,2$, the following relations are valid:
\begin{align}
\label{5B7}
&[\![x_{i,k+1}^{\pm},\,x_{j,l}^{\pm}]\!]
+[\![x_{j,l+1}^{\pm},\,x_{i,k}^{\pm}]\!]=0, \quad k,l\in\mathbb{Z},\\
\label{5B7.1}
&[x_{i,k+1}^{+},\,x_{j,l}^{-}]=0, \quad k,l\in\mathbb{Z}.
\end{align}

Define
\begin{gather*}
\mathfrak{X}^{\pm}(i,j; k,l)\doteq [\![x_{i,k+1}^{\pm},\,x_{j,l}^{\pm}]\!]
+[\![x_{j,l+1}^{\pm},\,x_{i,k}^{\pm}]\!],
\end{gather*}
and compute the commutator of $\mathfrak{X}^{\pm}(i,j; k,l)$ with $h_{i,1}$ and independently with $h_{j,1}$. Consequently, if $\mathfrak{X}^{\pm}(i,j; k,l)=0$, then $\mathfrak{X}^{\pm}(i,j; k+1,l)$ and $\mathfrak{X}^{\pm}(i,j; k,l+1)$ are solutions to the homogeneous system of two equations with the determinant
$$
\left|
\begin{array}{cc}
~ [a_{ii}]&[a_{ij}]\\
~[a_{ji}]&[a_{jj}]
\end{array}
\right|\neq 0.
$$

Hence, it holds that
\begin{gather*}
\mathfrak{X}^{\pm}(i,j; k,l)=0 \Rightarrow \mathfrak{X}^{\pm}(i,j; k+1,l)=0,\quad \mathfrak{X}^{\pm}(i,j; k,l+1)=0.
\end{gather*}
As $\mathfrak{X}^{\pm}(i,j; 0,0)=0$ (see \textbf{\uppercase{\romannumeral2}}), we deduce \eqref{5B7} for all $k,l$. A similar calculation holds for \eqref{5B7.1}.

\textbf{\uppercase{\romannumeral4}}) We introduce the following relations:
\begin{align}
\label{5B8}
&[a_{i,s},\,a_{j,l}]=\delta_{s,-l}\frac{[s A_{ij}]_{i}}{s}\cdot\frac{q^{s c}-q^{-sc}}{q_j-q_j^{-1}},~~i,j=1,2,~~s,l\in\mathbb{Z}, \\
\label{5B9}
&q^{\frac{s-l}{2}c}[x_{i,s}^+,\,x_{i,l}^-]=
q^{\frac{s+l}{2}c}[x_{i,s+l}^+,\,x_{i,0}^-],~~i=1,2,~~s,l\in\mathbb{Z},\\
\label{5B10}
&[\![x_{2,k+1}^{\pm},\,x_{2,l}^{\pm}]\!]
+[\![x_{2,l+1}^{\pm},\,x_{2,k}^{\pm}]\!]=0, \quad [x_{1,k}^{\pm},\,x_{1,l}^{\pm}]=0,\quad k,l\in\mathbb{Z},\\
\label{5B10.1}
&[a_{i,\pm s},\,x_{j,k}^{\pm}]=\pm\frac{[sA_{ij}]_i}{s}q^{\mp\frac{sc}{2}}x_{j,k+s}^{\pm},\quad i,j=1,2,~~s\in \mathbb{Z}_{>0},~~k\in\mathbb{Z}.
\end{align}

First, let us compute
\begin{gather*}
[\,a_{2,\pm1},[\,a_{2,\pm1},\,[\![\,x_{i,\pm k\pm1}^{\pm},\,x_{i,\pm k}^{\pm}\,]\!]] \quad\textrm{and} \quad [\,a_{2,\pm2},\,[\![\,x_{i,\pm k\pm1}^{\pm},\,x_{i,\pm k}^{\pm}\,]\!]],
\end{gather*}
for $i=1,2$.

Consequently, $[\![\,x_{i,\pm k\pm2}^{\pm},\,x_{i,\pm k\pm1}^{\pm}\,]\!]=0$ for $[\![\,x_{i,\pm1}^{\pm},\,x_{i,0}^{\pm}\,]\!]=0$. This implies that \eqref{5B10} holds for all $s=l$ when $i=2$ and for all $s=l-1$ when $i=1$.

Next, we establish the validity of relation \eqref{5B8} for $i=j=2$ and relation \eqref{5B9} for $i=1,2$ through induction on $s+l=p>0$. The base cases for $p=1,2$ are evident from our previous results. Assuming these relations hold for $p=r$, we derive
\begin{align*}
\cdots&=q^{\frac{-rc}{2}}[x_{i,r}^+,\,x_{i,0}^-]=q^{\frac{(2-r)c}{2}}[x_{i,r-1}^+,\,x_{i,1}^-]=\cdots\\
&=q^{\frac{(r-2)c}{2}}[x_{i,1}^+,\,x_{i,r-1}^-]=q^{\frac{rc}{2}}[x_{i,0}^+,\,x_{i,r}^-]=\cdots.\nonumber
\end{align*}

Applying $a_{2,1}$ to the aforementioned series of equations, we derive
\begin{align}
\label{5B11}
\cdots&=q^{\frac{-(r+1)c}{2}}\big([x_{i,r+1}^+,\,x_{i,0}^-]-q^c[x_{i,r}^+,\,x_{i,1}^-]\big)\\\nonumber
&=q^{\frac{-(r-1)c}{2}}\big([x_{i,r}^+,\,x_{i,1}^-]-q^c[x_{i,r-1}^+,\,x_{i,2}^-]\big)\\\nonumber
&=\cdots\\\nonumber
&=q^{\frac{(r-1)c}{2}}\left([x_{i,1}^+,\,x_{i,r}^-]-q^c[x_{i,0}^+,\,x_{i,r+1}^-]\right)=\cdots.\nonumber
\end{align}
Considering relation \eqref{5B9} for $i=2$, based on the induction hypothesis, we conclude that
\begin{gather*}
-k_2^{-1}q^{\frac{c}{2}}[\widetilde{\Phi}_{2, r}^+,\,[x_{2,0}^+,\,x_{2,1}^-]]=[a_{2,1},\,a_{2,r}]=k_2^{-1}q^{\frac{rc}{2}}[a_{2,1},\,[x_{2,0}^+,\,x_{2,r}^-]]
\end{gather*}

Applying Lemma \ref{LM5.1}-(1) to the left-hand side iteratively, we obtain
\begin{align*}
-[rA_{22}]_2\left([x_{2,1}^+,\,x_{2,r}^-]-q^c[x_{2,0}^+,\,x_{2,r+1}^-]\right)
=[A_{22}]_2\left([x_{2,1}^+,\,x_{2,r}^-]-q^c[x_{2,0}^+,\,x_{2,r+1}^-]\right),
\end{align*}
which implies that
$[x_{2,1}^+,\,x_{2,r}^-]-q^c[x_{2,0}^+,\,x_{2,r+1}^-]=0$.
Hence, based on equation \eqref{5B11}, we can conclude that relation \eqref{5B9} holds for $p=r+1$ when $i=2$.

Furthermore, for any positive integer $t<r$, we observe that
\begin{align*}
[a_{2,s-t},\,a_{2,t+1}]&=[\widetilde{\Phi}_{2,s-t}^+,\,a_{2,t+1}]=q^{\frac{(t+1)c}{2}}[\widetilde{\Phi}_{2,s-t}^+,\,[x_{2,0}^+,\,x_{2,t+1}^-]]\\
&=\frac{[A_{22}(s-t)]_2}{[A_{22}]_2}q^{\frac{rc}{2}}[\widetilde{\Phi}_{2,1}^+,\,[x_{2,0}^+,\,x_{2,s}^-]]\\
&=[A_{22}(s-t)]_2q^{\frac{(s-1)c}{2}}
\left([x_{2,1}^+,\,x_{2,r}^-]-q^c[x_{2,0}^+,\,x_{2,s+1}^-]\right)=0.
\end{align*}
This implies that relation \eqref{5B8} holds for $p=r+1$ when $i=2$ and $s,l>0$. A corresponding result is observed for the case $s+l=p<0$ and $s,l<0$.

Applying Lemma \ref{LM5.1}-(2) and using \eqref{5B7}, we obtain
\begin{align}
\label{5B12}
&[a_{2,\pm s},\,x_{j,k}^{\pm}]=\pm\frac{[sA_{2j}]_2}{s}q^{\mp\frac{sc}{2}}x_{2,k+s}^{\pm},\quad j=1,2, s\in \mathbb{Z}_{>0},~k\in\mathbb{Z}.
\end{align}

Moreover, by applying $[a_{2, k},[a_{2, l}, .]]$ to $[\![\,x_{i,\pm1}^{\pm},\,x_{i,0}^{\pm}\,]\!]=0$, we can establish relation \eqref{5B10} for any $k,l\in\mathbb{Z}$ with $i=1,2$.

Now, let us proceed to complete relation \eqref{5B9} for $i=1$. By \eqref{5B2}, it is easy to observe that
\begin{align*}
  [a_{2,r-1},\,[x_{1,0}^+,\,x_{1,2}^-]]=[a_{2, r-1},a_{1, 2}]=0=[a_{2, r},a_{1, 1}]=[a_{2,r},\,[x_{1,0}^+,\,x_{1,1}^-]],
\end{align*}
hence $[x_{1,r}^+,\,x_{1,1}^-]-q^c[x_{1,r-1}^+,\,x_{1,2}^-]=0$, validating the claim for $p=r+1$. It holds similarly for the case $s+l=p<0$.

To establish the rest of relation \eqref{5B8}:
when $i=j=2$, we deduce it for all $s+l=0$ or $s-l\neq0$ ($s,-l\in \mathbb{Z}_{>0}$) from \eqref{5B9};
when $i=j=1$, we derive it for all $s,l\in \mathbb{Z}$ from \eqref{5B9} and Lemma \ref{LM5.1}-(1);
when $i\neq j$, it follows from \eqref{5B9} and \eqref{5B12}.
Subsequently, applying Lemma \ref{LM5.1}-(2) again completes relation \eqref{5B10.1} for $i=1$.

Now, the verification of relations in steps ({\bf\uppercase{\romannumeral1}}) and ({\bf\uppercase{\romannumeral3}})-({\bf\uppercase{\romannumeral4}}) implies the establishment of relations \eqref{DE1}-\eqref{DE7}.

\vspace{6pt}
(2) The Serre relations \eqref{DE8}-\eqref{DE9}
\vspace{6pt}

The Serre relations \eqref{DE8} resemble those discussed in \cite[Section 4]{LYZ} for type A. Here, our focus is primarily on the Serre relation \eqref{DE9} within $\mathcal{U}'$, specifically involving the long root generators, as depicted in \eqref{new-DE9}:

\begin{gather}\label{new-DE9}
\operatorname{Sym}_{k_1,k_2,k_3}[\![x_{2,k_1}^{\pm},\,[\![x_{2,k_2}^{\pm},[\![x_{2,k_3}^{\pm},\,X_{1,l}^{\pm}]\!] ]\!] ]\!]=0.
\end{gather}

Let us denote the left side of relation \eqref{new-DE9} as $\mathcal{Y}^{\pm}(k_1,k_2,k_3;l)$. As per \eqref{S7}, we have $\mathcal{Y}^{\pm}(0,0,0;0)=0$. By applying $[a_{2,r},\,\cdot\,]$ and $[a_{1,r},\,\cdot\,]$ to $\mathcal{Y}^{\pm}(0,0,0;s)=0$, we get

\begin{equation*}
\begin{cases}
\frac{3}{r}[r A_{22}]_{2}\mathcal{Y}^{\pm}(r,0,0;0)+\frac{1}{r}[r A_{21}]_{2}\mathcal{Y}^{\pm}(0,0,0;r)=0, \\
\frac{3}{r}[r A_{12}]_{1}\mathcal{Y}^{\pm}(r,0,0;0)+\frac{1}{r}[r A_{11}]_{1}\mathcal{Y}^{\pm}(0,0,0;r)=0.
\end{cases}
\end{equation*}
Since $(\alpha_{1},\alpha_{1})=0$, we immediately have $\mathcal{Y}^{\pm}(r,0,0;0)=\mathcal{Y}^{\pm}(0,0,0;r)=0$ for all $r\in\mathbb{Z}$. Using the action $[a_{2,r},\,\cdot\,]$ on $\mathcal{Y}^{\pm}(0,0,0;s)=0$ three times yields
\begin{gather*}
\mathcal{Y}^{\pm}(k_1,k_2,k_3;l)=\mathcal{Y}^{\pm}(k_1,k_2,0;l)=\mathcal{Y}^{\pm}(k_1,0,0;l)=0
\end{gather*}
for all $k_1,k_2,k_3,l\in\mathbb{Z}$.

\vspace{6pt}
Therefore, we have proved that the map $\Theta_2$ is an epimorphism. Now, it is easy to see that $\Theta_1\circ\Theta_2=1_{\mathcal{U}'}$. On the other hand, by the definition \eqref{S10}-\eqref{S14}, we find that $\Theta_2\circ\Theta_1=1_{\mathcal{U}_q[\mathfrak{osp}(3|2)^{(1)}]}$. Hence, we complete the proof of (a).

\vspace{12pt}
{\bf Proof of Theorem \ref{the5.1}-(b):}
As the maps $\Psi_1$ and $\Psi_2$ defined in this Section, we outline the proof in the following steps:

({\bf\uppercase{\romannumeral1}}) $\Psi_1$ is a homomorphism.

By the defining relations \eqref{S10}-\eqref{S14}
in $\mathcal{U}'$, \eqref{S9} is equal to
 \begin{gather*}
[\Phi_{i,\pm1}^{\pm}, \Phi_{i,\pm2}^{\pm}]=0,\quad i=1,2.
 \end{gather*}
{Note that
 \begin{gather*}
 \Psi_1(\Phi_{i,\pm r}^{\pm})=\pm(q_i-q_i^{-1})\Psi_1(q^{\frac{r}{2}})[\Psi_1(x_{i,0}^{\pm}),\,\Psi_1(x_{i,\pm r}^{\mp})], \quad \text{for}~~ i,r\in\{1,2\}.
 \end{gather*}By
the defining relations of (E4)-(E5) and \eqref{S13}, we have
$\widetilde{E}_{r\delta^{(i)}}=K_{\delta}^{\frac{r}{2}}K_i^{-1}(q_i-q_i^{-1})^{-1}\Psi_1(\Phi_{i,r}^+)$ and $\widetilde{F}_{r\delta^{(i)}}=K_{\delta}^{\frac{-r}{2}}K_i^{-1}(q_i^{-1}-q_i)^{-1}\Psi_1(\Phi_{i,-r}^-)$, where $i,r\in\{1,2\}$.  From Lemma \ref{LM:T6}, we find that
\begin{align}\label{useful:911}
   &[\Psi_1(\Phi_{i,1}^{+}), \Psi_1(\Phi_{i,2}^{+})]=[\widetilde{E}_{\delta^{(i)}},\,\widetilde{E}_{2\delta^{(i)}}]=[E_{\delta^{(i)}},\,E_{2\delta^{(i)}}]=0.
\end{align}}By Proposition \ref{anti-auto}, we also have
\begin{align}\label{useful:911-0}
   & [\Psi_1(\Phi_{i,1}^{-}), \Psi_1(\Phi_{i,2}^{-})]=[\widetilde{F}_{\delta^{(i)}},\,\widetilde{F}_{2\delta^{(i)}}]=[F_{\delta^{(i)}},\,F_{2\delta^{(i)}}]=0.
\end{align}

We now prove that the new relation \eqref{S9} preserves the homomorphism relation. Consequently, by Lemma \ref{LM:T1}-\ref{LM:T3},  Lemma \ref{LM:T5} and Proposition \ref{anti-auto}, the mapping $\Psi_1$
is necessarily a homomorphism.

\vspace{6pt}
({\bf\uppercase{\romannumeral2}}) $\Psi_2$ is a homomorphism.

First, we examine relations \eqref{DJ1}--\eqref{DJ10} involving the elements $K_0$ and $\chi_0^{\pm}$. Relations \eqref{DJ1}--\eqref{DJ2} are evident by definition. Our focus then shifts to confirming relations \eqref{DJ3}--\eqref{DJ10}.
We will address the cases of $\chi_0^+$ and $\chi_0^-$ separately, starting with $\chi_0^+$. The case of $\chi_0^-$ follows a similar argument. To begin, we establish relation \eqref{DJ3} by considering the illustrative example of $(m,n)=(1,1)$ since the general case of $m, n\in\mathbb{Z}$ follows a similar approach. For $i=0$, one obtains
\begin{align*}
[\Psi_2(\chi_0^+),\,\Psi_2(\chi_0^-)]
&=\nu_0^+\nu_0^-[\,[\![x_{1,1}^-,x_{2,0}^-,x_{2,0}^-,x_{1,0}^-]\!]_r,\,[\![x_{1,-1}^+,x_{2,0}^+,x_{2,0}^+,x_{1,0}^+]\!]_r] \\
&=-\nu_0^+\nu_0^-(-1)^{[\alpha_1]+[\alpha_2]}q^{d_2+2d_1}[2]_2^2[2]_1
\frac{q^{-c}k_1^2k_2^2-q^ck_1^{-2}k_2^{-2}}{q_1-q_1^{-1}} \\
&=[2]_1\frac{\Psi_2(K_0)-\Psi_2(K_0^{-1})}{q_1-q_1^{-1}}.
\end{align*}
For $i=1$, we have
\begin{align*}
[\Psi_2(\chi_0^+),\,\Psi_2(\chi_1^-)]
&=[\nu_0^+\,\Psi_2(K_0)\,[\![x_{1,1}^-,x_{2,0}^-,x_{2,0}^-,x_{1,0}^-]\!]_r,\,x_{1,0}^-] \\
&=\nu_0^+\,\Psi_2(K_0)\,[\,[[\![x_{1,1}^-,x_{2,0}^-]\!],\,[x_{2,0}^-,x_{1,0}^-]_{q}]_{q},\,x_{1,0}^- ]_{q^2} \\
&=\nu_0^+\,\Psi_2(K_0)\,[\,[\![x_{1,1}^-,x_{2,0}^-]\!],\,[[x_{2,0}^-,x_{1,0}^-]_{q},\,x_{1,0}^-]_{q},\,x_{1,0}^- ]_{q^2} \\
&+(-1)^{[\alpha_1]}q\,[\,[[\![x_{1,1}^-,x_{2,0}^-]\!],\,x_{1,0}^-]_{q},\,  [x_{2,0}^-,\,x_{1,0}^-]_{q}\,]=0.
\end{align*}
As for $i=2$, applying $[\,\cdot\,,\,x_{1,0}^-]_{q}$ to $[\Psi_2(\chi_0^+),\,\Psi_2(\chi_2^-)]$, we obtain
\begin{align*}
[[\Psi_2(\chi_0^+)&,\,\Psi_2(\chi_2^-)]\,x_{1,0}^-]_{q}
=\nu_0^+\,\Psi_2(K_0)\,[\![x_{1,1}^-,x_{2,0}^-,x_{2,0}^-,x_{1,0}^-,x_{2,0}^-,x_{1,0}^-]\!]_{r} \\
&=\nu_0^+\,\Psi_2(K_0)\,[[\![x_{1,1}^-,x_{2,0}^-,x_{2,0}^-,x_{1,0}^-]\!]_r,\,[x_{2,0}^-,\,x_{1,0}^-]_q]_q
+[[\Psi_2(\chi_0^+),\,\Psi_2(\chi_1^-)],\,x_{2,0}^-]_{q^{-1}} \\
&=\nu_0^+\,\Psi_1(K_0)\,\Big[\Big[[\![x_{1,1}^-,x_{2,0}^-]\!],[x_{2,0}^-,x_{1,0}^-]_q\Big]_q,\,[x_{2,0}^-,\,x_{1,0}^-]_q\Big]_q=0.
\end{align*}
Then $[\Psi_2(\chi_0^+),\,\Psi_2(\chi_2^-)]=0$ can be obtained by applying $[x_{1,0}^+,\,\cdot\,]_{q^2}$ to both sides of above equality.

Next, we check relation \eqref{DJ4}. One gets for $i=2$
\begin{align*}
[\Psi_2(\chi_0^+),\,\Psi_2(\chi_2^+)]
&=\nu_0^+\,\Psi_2(K_0)\,[[\![x_{1,1}^-,x_{2,0}^-,x_{2,0}^-,x_{1,0}^-]\!]_r,\,x_{2,0}^+ ]=0.
\end{align*}

Now, concurrently, we proceed to verify the Serre relation involving the element $\chi_0^+$.
Relation \eqref{DJ5} ensues from $[\Psi_2(\chi_0^+),\,x_{1,1}^-]=0$, which can be confirmed as $[\Psi_2(\chi_0^+),\,\Psi_2(\chi_1^-)]=0$.
As for relation \eqref{DJ7}, we obtain
\begin{align*}
&[\![\Psi_2(\chi_1^+),\Psi_2(\chi_1^+),\Psi_2(\chi_1^+),\Psi_2(\chi_0^+)]\!]_{\ell} \\
=&[\![x_{1,0}^+,x_{1,0}^+,x_{1,0}^+,\nu_0^+\,\Psi_2(K_0)[\![x_{1,1}^-,x_{2,0}^-,x_{2,0}^-,x_{1,0}^-]\!]_r]\!]_{\ell} \\
=&p\,\Psi_2(K_0)k_1^2[\![x_{2,1}^-,x_{2,0}^-,x_{2,0}^-]\!]_r=0,
\end{align*}
where $p\in\mathbb{C}(q)$.

Let us examine the left-hand side of relation \eqref{DJ10}:
\begin{align*}
\textbf{LHS}&=[\![\, [\![x_{2,0}^+,x_{1,0}^+]\!] ,\, [\![x_{2,0}^+,x_{1,0}^+]\!],\,[\![x_{2,0}^+,x_{1,0}^+]\!],\,\nu_0^+\Psi_2(K_0)[\![x_{1,1}^-,x_{2,0}^-,x_{2,0}^-,x_{1,0}^-]\!]_r \,]\!]_{\ell} \\
&=\nu_0^+[\![\, [\![x_{2,0}^+,x_{1,0}^+]\!],\,[\![x_{2,0}^+,x_{1,0}^+]\!],\,-q^ck_1^{-1}k_2^{-1}q^{-1}[2]_1[2]_2\,[\![x_{1,1}^-,\,x_{2,0}^-]\!]  \,]\!]_{\ell} \\
&=\nu_0^+[\![\, [\![x_{2,0}^+,x_{1,0}^+]\!],\,-[2]_1[2]_2q^{\frac{1}{2}c}
(q^{-1}a_{2,1}+[2]_2a_{1,1}+k_1^{-1}[2]_2(q^{-2}-1)q^{\frac{1}{2}c}x_{1,1}^-x_{1,0}^+)\,]\!]_{\ell} \\
&=-\nu_0^+q^{-1}[2]_1[2]_2^2
([x_{2,0}^+,\,x_{1,1}^+]_q-[x_{2,1}^+,\,x_{1,0}^+]_q)-[2]_1[2]_2^2([x_{2,1}^+,\,x_{1,0}^+]_q+(q^{-2}-1)x_{2,1}^+x_{1,0}^+) \\
&=\nu_0^+q^{-1}[2]_1[2]_2^2(1-q-q^{-1})[x_{2,1}^+,\,x_{1,0}^+]_q,
\end{align*}
Its right hand side is as follows:
\begin{align*}
\textbf{RHS}&=(1-[2]_1)[\![\,[\![x_{2,0}^+,x_{1,0}^+]\!] ,\, [\![x_{2,0}^+,\,x_{2,0}^+,\,x_{1,0}^+,\,\nu_0^+\Psi_2(K_0)[\![x_{1,1}^-,x_{2,0}^-,x_{2,0}^-,x_{1,0}^-]\!]_r\,]\!]_{\ell},\,x_{1,0}^+\,]\!]_r \\
&=\nu_0^+(1-[2]_1)[\![\,[\![x_{2,0}^+,x_{1,0}^+]\!] ,\, [\![x_{2,0}^+,\,x_{2,0}^+,\,q^ck_1^{-1}k_2^{-2}q^{-2}[2]_1[\![x_{1,1}^-,x_{2,0}^-,x_{2,0}^-]\!]_r    \,]\!]_{\ell},\,x_{1,0}^+\,]\!]_r \\
&=\nu_0^+(1-[2]_1)[\![\,[\![x_{2,0}^+,x_{1,0}^+]\!] ,\, [\![x_{2,0}^+,\,-q^ck_1^{-1}k_2^{-1}q^{-1}[2]_1[2]_2[\![x_{1,1}^-,x_{2,0}^-]\!]    \,]\!],\,x_{1,0}^+\,]\!]_r \\
&=\nu_0^+(1-[2]_1)[\![\,[\![x_{2,0}^+,x_{1,0}^+]\!] ,\, q^ck_1^{-1}q^{-1}[2]_1[2]_2^2x_{1,1}^-,\,x_{1,0}^+\,]\!]_r \\
&=\nu_0^+(1-[2]_1)[\![\,q^{-1}[2]_1[2]_2^2x_{2,1}^+,\,x_{1,0}^+\,]\!],
\end{align*} which implies that relation \eqref{DJ10} with $(m,n)=(1,1)$ holds.
%%%%%%%%%%%%%%%%%%%%%%%%%%%%%%%%%%%%%%%%%%%%%%%%%%%%%%%%%%%%%%%%%%%%%%%%%%%%%%%%%%%%%%%%%%%%%%%%%%%%%%%%%%%%%%%%%%%%%%%%%%%%%

Finally, it is noteworthy that relations \eqref{DJ1}-\eqref{DJ10}, excluding the elements $K_0$ and $\chi_0^{\pm}$, manifest a consistent and straightforward pattern.

\vspace{6pt}
({\bf\uppercase{\romannumeral2}}) $\Psi_1\circ\Psi_2=\operatorname{id}_{U_q[\mathfrak{osp}(3|2)^{(1)}]}$ and $\Psi_2\circ\Psi_1=\operatorname{id}_{\mathcal{U}'}$.

{The homomorphisms} $\Psi_1$ and $\Psi_2$ are evidently surjective by construction. To establish this, we will scrutinize their action on the generators. It is worth noting that the elements $\chi_i^{\pm}$, $K_{\alpha}$ (or alternatively, $\chi_i^{\pm}$, $K_i^{\pm 1}$, $K_{\delta}^{\pm\frac{1}{2}}$) for $i=1,2$ and $\alpha\in\widehat{Q}$ remain fixed under the composition $\Psi_1\circ\Psi_2$ (or alternatively, $\Psi_2\circ\Psi_1$).

Next, we assert that $\Psi_1\circ\Psi_2(\chi_0^+)=\chi_0^+$. Indeed, from the Example \ref{Examp},
\begin{align*}
\Psi_1\circ\Psi_2(\chi_0^+)&=\Psi_1\left(\nu_0^+\,\Psi_2(K_0)\,
[\![x_{1,1}^-,x_{2,0}^-,x_{2,0}^-,x_{1,0}^-]\!]_r\right) \\
&=-[2]_2^{-1}K_0[\![K_1K_{\delta-\alpha_1}^{-1}E_{\delta-\alpha_1},\chi_{2}^-,\chi_{2}^-,\chi_1^-]\!]_r\\
&=[2]_2^{-1}[2]_1^{-1}K_0[\![K_1K_{\delta-\alpha_1}^{-1}[\![\chi_2^+,\chi_2^+,\chi_{1}^+,\chi_0^+]\!]_{\ell},
\chi_{2}^-,\chi_{2}^-,\chi_1^-]\!]_r.
\end{align*}A simple calculation shows that $\Psi_1\circ\Psi_2(\chi_0^+)=\chi_0^+$. Similarly, one can verify that $\Psi_1\circ\Psi_2(\chi_0^-)=\chi_0^-$ and $\Psi_2\circ\Psi_1(x_{i,\pm 1}^{\mp})=x_{i,\pm 1}^{\mp}$. Thus, the proof is complete.

{
\subsection{The case for $N>2$ of Theorem \ref{the5.1}}}

{\bf Proof of Theorem \ref{the5.1}:}
(a){
If $n=1$ and $m>1$, then we can complete the proof by iteratively applying the steps as well as the low rank case for the generators with index $i<m+1=N$ from the definition \eqref{S10}-\eqref{S14}, and the generator relations with indices $m$ and $m+1$ are analogous to those in the established low rank case for indices 1 and 2.
Obviously, the relations for the generators with indices $i<m+1$
are relatively easier than those for $i=m+1$.}

If $n>1$ and $m>1$, since $a_{11}\neq0$, we initiate relations \eqref{DE1}-\eqref{DE7} specifically for the index $i=j=1$ in $\mathcal{U}'$. Subsequently, to derive relations \eqref{DE1}-\eqref{DE7} for $i,j=2,3,...$, we also employ a step-by-step repetition of the methodology used in the low rank case. This approach closely follows the methodology outlined in \cite[Section 4]{LYZ} for type A.

 {Finally, we can conclude that the compositions $\Theta_1\circ\Theta_2$ and $\Theta_2\circ\Theta_1$ both equal the identity map on their respective domains.}

\vspace{6pt}
(b) The proof method for the case of $N>2$ is an analogous extension of the low-rank case. It is worth noting that to prove that $\Psi_2$ is a homomorphism,
an additional Serre relation \eqref{DJ9} must also be considered. Without loss of generality, let $(m,n)=(2,1)$, then
\begin{align*}
&[\![\Psi_2(\chi_3^+),\Psi_2(\chi_2^+),\Psi_2(\chi_1^+),\Psi_2(\chi_0^+),\Psi_2(\chi_1^+),
\Psi_2(\chi_2^+),\Psi_2(\chi_1^+)]\!]_r \\
=&[\![\, [\![x_{3,0}^+,x_{2,0}^+,x_{1,0}^+]\!]_r,\, \nu_0^+\Psi_2(K_0)[\![x_{1,1}^-,\ldots,x_{3,0}^-,x_{3,0}^-,\ldots,x_{1,0}^-]\!]_r,\,x_{1,0}^+,\,x_{2,0}^+,\,x_{1,0}^+ \,]\!]_r \\
=&[\![\,q^{-2}\nu_0^+\Psi_2(K_0)[\,[\![x_{3,0}^+,\,x_{2,0}^+]\!],\,[2]_1k_1 [\![x_{1,1}^-,x_{2,0}^-,x_{3,0}^-,x_{3,0}^-,x_{2,0}^-]\!]_r]_q,\,x_{1,0}^+,\,x_{2,0}^+,\,x_{1,0}^+ \,]\!]_r \\
=&[\![\,\nu_0^+q^ck_1^{-1}k_2^{-1}k_3^{-1}[2]_1[2]_3\,[\![x_{1,1}^-,\,x_{2,0}^-,x_{3,0}^-]\!]_r,\,  x_{1,0}^+,\,x_{2,0}^+,\,x_{1,0}^+ \,]\!]_r \\
=&[\![\,q^2\nu_0^+q^ck_2^{-1}k_3^{-1}[2]_1[2]_3\,[\![x_{2,1}^-,\,x_{3,0}^-]\!],\,x_{2,0}^+,\,x_{1,0}^+ \,]\!]_r \\
=&[\![\,q^2\nu_0^+q^ck_3^{-1}[2]_1[2]_3\,x_{3,1}^-  ,\,x_{1,0}^+ \,]\!]=0.
\end{align*}

On the other hand, {to prove that the homomorphism $\Psi_1\circ\Psi_2$
is an identity map for the general case, we still consider the element $\chi_0^+$,} that is
\begin{align*}
\Psi_1\circ\Psi_2&(\chi_0^+)=\Psi_1\left(\nu_0^+\,\Psi_2(K_0)\,
[\![x_{1,1}^-,x_{2,0}^-,\ldots,x_{N,0}^-,x_{N,0}^-,\ldots,x_{1,0}^-]\!]_r\right) \\
&=\nu_0^+\mu_1^+K_0[\![K_1K_{\delta-\alpha_1}^{-1}[\![\chi_2^+,\ldots,\chi_{N}^+,\chi_{N}^+,\ldots,\chi_0^+]\!]_{\ell},
\chi_2^+,\ldots,\chi_{N}^-,\chi_{N}^-,\ldots,\chi_1^-]\!]_r \\
&=[2]_1[2]_{N}^2\,q^{2(d_1+\cdots+d_{N})-d_2}\prod_{i=2}^{N-1}
\left(\frac{q^{d_i}-q^{-d_i}}{q_i-q_i^{-1}}\frac{q^{d_{i+1}}-q^{-d_{i+1}}}{q_i-q_i^{-1}}\right)\chi_0^+=\chi_0^+,
\end{align*}
where $\mu_1^+=-q^{(\alpha_1,\alpha_1)+2n+1-2m}[2]_N^{-1}[2]_1^{-1}$.

\vspace{12pt}
Now, we can give the proof of our main Theorem \ref{Main1}.

{\bf Proof of Theorem \ref{Main1}:} The proof follows directly from Theorem \ref{the5.1}, and it is evident that $\Psi=\Psi_1\circ\Theta_2$.

\begin{remark}
{It is worth noting that Xu-Zhang (\cite{YXRZ2})} established an isomorphism between the Drinfeld presentation and the Drinfeld-Jimbo presentation of the quantum affine superalgebra associated with the remaning case $\mathfrak{osp}(1|2n)^{(1)}$.
\end{remark}

\begin{appendix}
    \section{Appendix: Explicit braid action with Dynkin Diagrams}
    As known, the nodes, links, and arrows of Dynkin diagram for the orthosymplectic Lie superalgebra at arbitrary parities determine the actions of $r_i$. Thus, the action of $T_i$ on Chevalley generators depends on its corresponding Dynkin diagram. Since $\omega_i$ preserves the standard Dynkin diagram, we don't need to use all braid group action on quantum affine superalgebra $U_q[\mathfrak{osp}(2m+1|2n)^{(1)}]$ for arbitrary diagrams. We list only some of Bezerra-Futorny-Iryna's results we need. More details please see \cite[Section 4]{Luan}. Throughout this appendix, we use
    \begin{tikzpicture}
            \draw (0.16,0.16)--(-0.16,-0.16)  (0.16,-0.16)--(-0.16,0.16);
          \end{tikzpicture}
    to represent the white or gray node and use
    \begin{tikzpicture}
            \draw (0,0) circle (0.22);
            \draw[fill=black] (0,0) circle (0.12);
          \end{tikzpicture}
    to represent the white or black node. Set $d_i=(\varepsilon_i,\varepsilon_i)$ for $\{\varepsilon_i|i=1,\ldots,N\}$ of $\mathfrak{osp}(2m+1|2n)$.
    \begin{enumerate}
        \item\textbf{Case} 1. If the Dynkin diagram of $\mathfrak{osp}(2m+1|2n)^{(1)}$ has subdiagram as
       \begin{center}
          \begin{tikzpicture}
            \draw (0.22,0) circle (0.22);
            \node[above] at (0.22,0.4) {0};
            \draw[double, double distance=3pt] (0.44,0) -- (1.1,0);
            \draw[-{To[scale=2, line width=0.4pt]}, line width=0.4pt] (1.2,0) -- ++(0.01,0);
            \draw (1.46,0) circle (0.22);
            \node[above] at (1.46,0.4) {1};
            \draw (1.68,0)--(2.48,0);
            \draw (2.54,0.16)--(2.86,-0.16)  (2.54,-0.16)--(2.86,0.16);
            \node[above] at (2.70,0.4) {2};
            \draw[dash pattern=on 3pt off 2pt] (2.92,0)--(3.52,0);
          \end{tikzpicture}
        \end{center}
        then $T_0$ is an automorphism of superalgebra $U_q[\mathfrak{osp}(2m+1|2n)^{(1)}]$ given by
        \begin{equation}\label{Braid:1}
        \begin{split}
            &T_0(\chi_0^+)=-d_1(q+q^{-1})\chi_0^-K_0,\quad T_0(\chi_0^-)=-\frac{d_1}{q+q^{-1}}K_0^{-1}\chi_0,\quad T_0(K_0)=K_0^{-1}, \\
            &T_0(\chi_1^+)=\frac{d_1q^{-2d_1}}{q+q^{-1}}[\![\chi_1^+,\,\chi_0^+]\!],\quad T_0(\chi_1^-)=-[\![\chi_1^-,\,\chi_0^-]\!],\quad T_0(K_1)=K_0K_1;
        \end{split}
        \end{equation}
        while $T_1: U_q[\mathfrak{osp}(2m+1|2n)^{(1)}]\rightarrow U_q[^{r_1}\mathfrak{osp}(2m+1|2n)^{(1)}]$ is defined by
        \begin{equation}\label{Braid:2}
            \begin{split}
                &T_1(\chi_0^+)=\frac{q^{-2d_2}}{q+q^{-1}}[\![ [\![\chi_0^+,\,\chi_1^+]\!] ,\,\chi_1^+]\!],\quad T_1(\chi_0^-)=[\![ [\![\chi_0^-,\,\chi_1^-]\!],\,\chi_1^-]\!],\\
            &T_1(\chi_1^+)=-d_1\chi_1^-K_1,\quad T_1(\chi_1^-)=-d_{2}K_1^{-1}\chi_1^+, \\
                &T_1(\chi_2^+)=(-1)^{[\alpha_1][r_1(\alpha_2)]}d_1q^{-d_1}[\![\chi_2^+,\,\chi_1^+]\!],\quad T_1(\chi_2^-)=-[\![\chi_2^-,\,\chi_1^-]\!], \\
                & T_1(K_0)=K_0K_1^2,\quad T_1(K_1)=K_1^{-1},\quad T_1(K_2)=K_1K_2.
            \end{split}
        \end{equation}

        \item\textbf{Case} 2. If the Dynkin diagram of $\mathfrak{osp}(2m+1|2n)^{(1)}$ has subdiagram as
        \begin{center}
          \begin{tikzpicture}
            \draw (0.22,0) circle (0.22);
            \node[above] at (0.22,0.4) {0};
            \draw[double, double distance=3pt] (0.44,0) -- (1.1,0);
            \draw[-{To[scale=2, line width=0.4pt]}, line width=0.4pt] (1.2,0) -- ++(0.01,0);
            \draw (1.46,0) circle (0.22) (1.30,0.16)--(1.62,-0.16)  (1.30,-0.16)--(1.62,0.16);
            \node[above] at (1.46,0.4) {1};
            \draw (1.68,0)--(2.48,0);
            \draw (2.54,0.16)--(2.86,-0.16)  (2.54,-0.16)--(2.86,0.16);
            \node[above] at (2.70,0.4) {2};
            \draw[dash pattern=on 3pt off 2pt] (2.92,0)--(3.52,0);
          \end{tikzpicture}
        \end{center}
        then $T_0$ is an automorphism of superalgebra $U_q[\mathfrak{osp}(2m+1|2n)^{(1)}]$ defined as \eqref{Braid:1};
        while $T_1: U_q[\mathfrak{osp}(2m+1|2n)^{(1)}]\rightarrow U_q[^{r_1}\mathfrak{osp}(2m+1|2n)^{(1)}]$ acts on node 0 given by
        \begin{align*}
            T_1(\chi_0^+)=\frac{d_2q^{-2d_2}}{q+q^{-1}}[\![\chi_0^+,\,\chi_1^+]\!],\quad T_1(\chi_0^-)=[\![\chi_0^-,\,\chi_1^-]\!],\quad T_1(K_0)=K_0K_1,
        \end{align*}
        and acts on nodes 1 and 2 the same as \eqref{Braid:2}.

        \item\textbf{Case} 3. If the Dynkin diagram of $\mathfrak{osp}(2m+1|2n)^{(1)}$ has subdiagram as
        \begin{center}
             \begin{tikzpicture}
             %  һ  ͼ
             \draw (1.84,0.6) circle (0.22) (1.68,0.44) -- (2.00,0.76) (1.68,0.76) -- (2.00,0.44);
      \node at (1.32,0.6) {$0$};
      \draw (1.84,-0.6) circle (0.22) (1.68,-0.44) -- (2.00,-0.76) (1.68,-0.76) -- (2.00,-0.44);
      \node at (1.32,-0.6) {$1$};
      \draw[dash pattern=on 3pt off 2pt]  (1.94,0.38)--(1.94,-0.38) (1.74,0.38)--(1.74,-0.38); %  Զ         ʽ
    \draw (3.02,0) circle (0.22);
    \node at (3.42,0.4) {$2$};
    \draw (2.06,0.6)--(2.80,0)  (2.06,-0.6)--(2.80,0);
    \draw[dash pattern=on 3pt off 2pt] (3.24,0)--(3.84,0);
    \node at (5,0) {or};
             %%%%%%%%%%%%%%%%%%%%%%%%%%%%%%%%%%%%%%%%%%%%%%%%%%%%%%%%%%%%%%
             % ڶ   ͼ
            \draw (6.84,0.6) circle (0.22) (6.68,0.44) -- (7.00,0.76) (6.68,0.76) -- (7.00,0.44);
      \node at (6.32,0.6) {$0$};
      \draw (6.84,-0.6) circle (0.22) (6.68,-0.44) -- (7.00,-0.76) (6.68,-0.76) -- (7.00,-0.44);
      \node at (6.32,-0.6) {$1$};
      \draw[dash pattern=on 3pt off 2pt]  (6.94,0.38)--(6.94,-0.38) (6.74,0.38)--(6.74,-0.38); %  Զ         ʽ
    \draw (7.86,0.16)--(8.18,-0.16)  (7.86,-0.16)--(8.18,0.16);
    \node at (8.42,0.4) {$2$};
    \draw (7.06,0.6)--(7.80,0)  (7.06,-0.6)--(7.80,0);
    \draw[dash pattern=on 3pt off 2pt] (8.24,0)--(8.84,0);
        \end{tikzpicture}
        \end{center}
        then $T_2: U_q[\mathfrak{osp}(2m+1|2n)^{(1)}]\rightarrow U_q[^{r_2}\mathfrak{osp}(2m+1|2n)^{(1)}]$ acts on node 0 and 1 given by
        \begin{align*}
            &T_2(\chi_0^+)=d_3q^{-d_3}[\![\chi_0^+,\,\chi_2^+]\!],\quad T_2(\chi_0^-)=-(-1)^{[r_2(\alpha_0)][\alpha_2]}[\![\chi_0^-,\,\chi_2^-]\!],\quad T_2(K_0)=K_0K_2, \\
            &T_2(\chi_1^+)=d_3q^{-d_3}[\![\chi_1^+,\,\chi_2^+]\!],\quad T_2(\chi_1^-)=-(-1)^{[r_2(\alpha_1)][\alpha_2]}[\![\chi_1^-,\,\chi_2^-]\!],\quad T_2(K_1)=K_1K_2.
        \end{align*}

       \item\textbf{Case} 4. If the Dynkin diagram of $\mathfrak{osp}(2m+1|2n)^{(1)}$ has subdiagram as
        \begin{center}
            \begin{tikzpicture}
                \draw[dash pattern=on 3pt off 2pt] (-0.82,0)--(-0.22,0);
                \draw (0.16,0.16)--(-0.16,-0.16)  (0.16,-0.16)--(-0.16,0.16);
                \node[above] at (0,0.4) {$i-1$};
                \draw  (0.22,0)--(1.22,0);
                \draw (1.28,0.16)--(1.60,-0.16)  (1.28,-0.16)--(1.60,0.16);
                \node[above] at (1.44,0.4) {$i$};
                \draw (1.66,0)--(2.66,0);
                \draw (2.72,0.16)--(3.04,-0.16)  (2.72,-0.16)--(3.04,0.16);
                \node[above] at (2.88,0.4) {$i+1$};
                \draw[dash pattern=on 3pt off 2pt] (3.10,0)--(3.70,0);
            \end{tikzpicture}
        \end{center}
        then $T_i: U_q[\mathfrak{osp}(2m+1|2n)^{(1)}]\rightarrow U_q[^{r_i}\mathfrak{osp}(2m+1|2n)^{(1)}]$ is defined by
        \begin{equation*}
            \begin{split}
                &T_i(\chi_{i-1}^+)=d_{i+1}q^{-d_{i+1}}[\![\chi_{i-1}^+,\,\chi_i^+]\!],\quad T_i(\chi_{i-1}^-)=-(-1)^{[r_i(\alpha_{i-1})][\alpha_i]}[\![\chi_{i-1}^-,\,\chi_i^-]\!], \\
                &T_i(\chi_i^+)=-d_i\chi_i^-K_i,\quad T_i(\chi_i^-)=-d_{i+1}K_i^{-1}\chi_i^+, \\
                &T_i(\chi_{i+1}^+)=(-1)^{[\alpha_i][r_i(\alpha_{i+1})]}d_{i}q^{-d_{i}}[\![\chi_{i+1}^+,\,\chi_i^+]\!],\quad T_i(\chi_{i+1}^-)=-[\![\chi_{i+1}^-,\,\chi_i^-]\!], \\
                & T_i(K_{i-1})=K_{i-1}K_i,\quad T_i(K_i)=K_i^{-1},\quad T_i(K_{i+1})=K_iK_{i+1}.
            \end{split}
        \end{equation*}

        \item\textbf{Case} 5. If the Dynkin diagram of $\mathfrak{osp}(2m+1|2n)^{(1)}$ has subdiagram as
        \begin{center}
          \begin{tikzpicture}
            \draw (0.06,0.16)--(0.38,-0.16)  (0.06,-0.16)--(0.38,0.16);
            \node[above] at (0.22,0.4) {$i$};
            \draw[double, double distance=3pt] (0.44,0) -- (1.1,0);
            \draw[-{To[scale=2, line width=0.4pt]}, line width=0.4pt] (1.2,0) -- ++(0.01,0);
            \draw (1.46,0) circle (0.22);
            \draw[fill=black] (1.46,0) circle (0.12);
            \node[above] at (1.46,0.4) {$j$};
          \end{tikzpicture}
        \end{center}
        then $T_i: U_q[\mathfrak{osp}(2m+1|2n)^{(1)}]\rightarrow U_q[^{r_i}\mathfrak{osp}(2m+1|2n)^{(1)}]$ acts on node $j$ given by
        \begin{align*}
          &T_i(\chi_j^+)=(-1)^{[\alpha_i][r_i(\alpha_j)]}d_iq^{-\frac{d_i+d_j}{2}}[\![\chi_j^+,\,\chi_i^+]\!], \\
          &T_i(\chi_j^-)=-q^{-\frac{d_j-d_i}{2}}[\![\chi_j^-,\,\chi_i^-]\!],\quad T_i(K_j)=K_iK_j;
        \end{align*}
        while $T_j$ is an automorphism of $U_q[\mathfrak{osp}(2m+1|2n)^{(1)}]$ action on nodes $i$ and $j$ given by
        \begin{align*}
            &T_j(\chi_j^+)=-\chi_j^-K_j,\quad T_j(\chi_j^-)=-(-1)^{[\alpha_j]}K_j^{-1}\chi_j^+,\quad T_j(K_j)=K_j^{-1}, \\
            &T_j(\chi_i^+)=q^{-d_j}[\![ [\![\chi_{i}^+,\,\chi_j^+]\!],\,\chi_j^+ ]\!],\quad T_j(\chi_i^-)=[\![ [\![\chi_{i}^-,\,\chi_j^-]\!],\,\chi_j^- ]\!],\quad T_j(K_i)=K_iK_j^2.
        \end{align*}
    \end{enumerate}

\end{appendix}

\section*{Declarations}

\textbf{Ethical Approval}
{The work does not involve human subjects, animal experiments, or any other activities requiring ethical approval.}\\

\noindent\textbf{Funding}
H. Zhang is supported by the support of the National Natural Science Foundation of China (No. 12271332), and Natural Science Foundation of Shanghai Municipality (No. 22ZR1424600)\\

\noindent\textbf{Data Availability} Data sharing is not applicable to this article as no datasets were generated or analyzed
during the current study\\

\noindent\textbf{Conflict of Interests} There are no conflicts of interest for this work.

\end{document}